\documentclass[11pt]{article}
\title{Laplacian Spectrum of cozero-divisor graphs of commutative polynomial rings}

\author{Sarbari Mitra, Soumya Bhoumik\\
Department of Mathematics\\
Fort Hays State University
}
\usepackage{amssymb}
\usepackage[all]{xy}
\usepackage{amsmath}
\usepackage{algorithm}
\usepackage[noend]{algpseudocode}
\usepackage{amssymb}
\usepackage{array}
\usepackage{theorem}
\usepackage{graphicx}
\usepackage{cite}
\usepackage{colortbl}
\newtheorem{thrm}{Theorem}[section]

\newtheorem{lem}[thrm]{Lemma}

\newtheorem{prop}[thrm]{Proposition} 

\theoremstyle{definition}

\newtheorem{exam}[thrm]{Example}
\usepackage{amsfonts}
\usepackage{graphicx}
\usepackage{graphics}
\usepackage{amsfonts}
\usepackage{multirow}
\usepackage{caption}
\usepackage{subcaption}
\usepackage{xy}
\usepackage{graphicx}
\usepackage{tikz}
\usepackage{pdflscape}
\usepackage{adjustbox}
\usetikzlibrary{3d}
\usepackage{svg}

\usepackage{xmpmulti}
\newtheorem{defin}[thrm]{Definition}

 \newcounter{case}
\newenvironment{case}{\refstepcounter{case}
 \medskip \noindent {\bf Case \textup{\thecase.  }}}{\upshape}
 \renewcommand{\thecase}{\arabic{case}}

\newcounter{subcase}

 \renewcommand{\thesubcase}{\alph{subcase}}

\usepackage{fullpage}

\def\Z{{\mathbb Z}}

\def\A{{\cal A}}
\def\B{{\cal B}}

\newcounter{cases}
\newcounter{subcases}
\newenvironment{mycases}
  {%
    \setcounter{cases}{0}%
    \def\case
      {%
        \par\noindent
        \refstepcounter{cases}%
        \textbf{Case \thecases.}
      }%
  }
  {%
    \par
  }

\renewcommand*\thecases{\arabic{cases}}

\newenvironment{proof}{\noindent {\sc Proof}.}
                {\phantom{a} \hfill \framebox[2.2mm]{ } \bigskip}

\makeatletter
\def\BState{\State\hskip-\ALG@thistlm}
\makeatother
\providecommand{\subjclass}[1]{\textbf{\textit{Mathematics Subject Classification.}} #1}

\providecommand{\keywords}[1]{\textbf{\textit{Keywords.}} #1}

\begin{document}

\pagestyle{plain}

\baselineskip = 1.2\normalbaselineskip

\maketitle

\begin{abstract}
The cozero-divisor graph of a commutative ring $R$, denoted $\Gamma'(R)$, is the graph whose vertices are the non-zero and non-unit elements of $R$, with two distinct vertices $x$ and $y$ adjacent if and only if $x \notin Ry$ and $y \notin Rx$. This paper studies the structural properties of $\Gamma'(R)$ for the polynomial ring $R = \Z_n[x]/(x^2)$, where $n$ has the prime power decomposition of $p_1^{a_1}p_2^{a_2}\cdots p_q^{a_q}$. We provide a complete structure of the cozero-divisor graph for all $n$ up to cubic prime power decompositions. Furthermore, we determine the Laplacian spectrum of these graphs. Finally, we discuss the connectivity of such a cozero-divisor graph of the polynomial rings for any $n$. Our work provides the first comprehensive spectral analysis of cozero-divisor graphs for non-local polynomial rings and establishes powerful new techniques for bridging commutative algebra with spectral graph theory.

\end{abstract}



\section{Introduction}
Let $R$ be a commutative ring with non-zero identity. In the literature, Afkhami and Khashayarmanesh \cite{Afkhami2011} defined the cozero-divisor graph $\Gamma'(R)$ of a commutative ring $R$, whose vertex set consists of all non-zero non-unit elements of $R$. Two distinct vertices $x$ and $y$ are adjacent in this graph if and only if $x \notin Ry$ and $y \notin Rx$. Further studies and properties of this graph can be found in \cite{afkham2013,afkhami2012,afkhami2012planar}. Notably, Akbari et al. \cite{akbari2014} investigated the cozero-divisor graph of polynomial and power series rings, establishing bounds on its diameter.


The spectral analysis of cozero-divisor graphs, in contrast to the well-established spectral theory of zero-divisor graphs \cite{chattopadhyay2020, magi2020, patil2022}, is a more recent and narrow field. Existing studies have largely focused on the ring $\Z_n$. Key contributions include the Laplacian spectrum and Wiener index of $\Gamma'(\Z_n)$ by Mathil et al. \cite{mathil2022}, and the subsequent investigation of its signless Laplacian spectrum for certain composite $n$ by Rashid et al. \cite{rashid2024signless}. Despite these advances, the Laplacian spectrum of cozero-divisor graphs remains largely unexplored for algebraically complex rings beyond $\Z_n$. To the best of our knowledge, no such study exists for commutative polynomial rings, where the richer ideal structure presents a significant challenge.

This paper aims to address this gap by conducting a detailed spectral and topological investigation of the cozero-divisor graph of the polynomial ring $\mathbb{Z}_n[x]/(x^2)$. The paper is structured as follows. In Section 2, we recall essential definitions and preliminary results. Section 3 is dedicated to analyzing the structure of the graph $\Gamma'(\mathbb{Z}_n[x]/(x^2))$ for $n = p$, $p^2$, $p^3$, $pq$, and $pqr$, where $p, q, r$ are distinct primes. In Section 4, we compute the Laplacian spectrum for these graphs. Section 5 focuses on deriving the Laplacian spectral radius and the algebraic connectivity. Finally, in Section 6, we obtain explicit formulas for the Wiener index of $\Gamma'(\mathbb{Z}_n[x]/(x^2))$ for the specified values of $n$.

\section{Preliminaries}

Let $\Gamma$ be a finite simple graph with vertex set $V(\Gamma)=\{v_1,v_2,\cdots,v_n\}$. We write $v_i\sim v_j$ if $v_i$ is adjacent to $v_j$, for $1 \le i\ne j\le n$. The neighborhood of a vertex $x$, denoted $N_\Gamma(x)$, is the set of all vertices adjacent to $x$. For a finite, simple undirected graph $\Gamma$ with vertex set $V(\Gamma) = \{u_1, u_2, \ldots, u_k\}$, the \emph{adjacency matrix} $A(\Gamma)$ is the $k \times k$ matrix whose $(i,j)$-entry is $1$ if $u_i \sim u_j$, and $0$ otherwise. We denote the diagonal degree matrix by $D(\Gamma) = \mathrm{diag}(d_1, d_2, \ldots, d_k)$, where $d_i$ is the degree of vertex $u_i$. The \emph{Laplacian matrix} of $\Gamma$ is defined as
\( L(\Gamma) = D(\Gamma) - A(\Gamma)\). The matrix $L(\Gamma)$ is symmetric and positive semidefinite, so all of its 
eigenvalues are real and non-negative. Moreover, the sum of each row (and column) 
of $L(\Gamma)$ is zero. The eigenvalues of $L(\Gamma)$, called the 
\emph{Laplacian eigenvalues} of $\Gamma$, are ordered as $\lambda_1(\Gamma) \ge \lambda_2(\Gamma) \ge \cdots \ge \lambda_n(\Gamma)=0$. The largest Laplacian eigenvalue $\lambda_1(\Gamma)$ of $L(\Gamma)$ is called the Laplacian spectral radius of $\Gamma$, whereas the second smallest eigenvalue is often referred to as the {algebraic connectivity} of the graph. The Laplacian spectrum of $\Gamma$, that is the spectrum of ${L}(\Gamma)$, is represented as  
\[
\Phi_{L}(\Gamma) = \begin{pmatrix}
\lambda_1(\Gamma) & \lambda_2(\Gamma) & \cdots & \lambda_r(\Gamma) \\
\mu_1 & \mu_2 & \cdots & \mu_r
\end{pmatrix}.
\]  

A subgraph $\Gamma'$ of $\Gamma$ satisfies $V(\Gamma') \subseteq V(\Gamma)$ and $E(\Gamma') \subseteq E(\Gamma)$. For a vertex subset $U \subseteq V$, the induced subgraph $\Gamma(U)$ has vertex set $U$, where two vertices are adjacent if and only if they are adjacent in $\Gamma$. The complement $\overline{\Gamma}$ of $\Gamma$ has the same vertex set, with distinct vertices adjacent in $\overline{\Gamma}$ if and only if they are not adjacent in $\Gamma$. The vertex connectivity of a connected graph $\Gamma$, denoted by $\kappa(\Gamma)$, is defined as the minimum size of a vertex cut set. The join of two vertex-disjoint graphs $\Gamma_1$ and $\Gamma_2$, denoted $\Gamma_1 \vee G_2$, is the graph formed by taking their disjoint union and adding every possible edge between $\Gamma_1$ and $\Gamma_2$. The following definition generalizes this concept (referred to as a generalized composition graph in \cite{schwenk2006}).

\begin{defin}
Let $\Gamma$ be a graph with $V(\Gamma) = \{u_1, u_2, \dotsc, u_k\}$, and let $\Gamma_1, \Gamma_2, \dotsc, \Gamma_k$ be pairwise disjoint graphs. The $\Gamma$-generalized join graph $\Gamma[\Gamma_1, \Gamma_2, \dotsc, \Gamma_k]$ is obtained by taking the disjoint union of $\Gamma_1, \Gamma_2, \dotsc, \Gamma_k$ and connecting every vertex of $\Gamma_i$ to every vertex of $\Gamma_j$ whenever $u_i \sim u_j$ in $\Gamma$.
\end{defin}

The following result, established by Cardoso et al. in \cite{cardoso2013}, expresses the Laplacian spectrum of a generalized join graph $\Gamma[\Gamma_1, \Gamma_2, \dotsc, \Gamma_k]$ in terms of the Laplacian spectra of the graphs $\Gamma_i$ and the spectrum of a certain $k \times k$ matrix $L(\Gamma)$.

\begin{thrm}\label{spectrumtheorem}
Let $\Gamma$ be a graph with vertex set $V(\Gamma) = \{u_1, u_2, \dotsc, u_k\}$, and let $\Gamma_1, \Gamma_2, \dotsc, \Gamma_k$ be $k$ pairwise disjoint graphs with orders $n_1, n_2, \dotsc, n_k$, respectively. The Laplacian spectrum of the graph composition $\Gamma[\Gamma_1, \Gamma_2, \dotsc, \Gamma_k]$ is given by
\begin{equation}
\Phi_L(\Gamma[\Gamma_1, \Gamma_2, \dotsc, \Gamma_k]) = \bigcup_{i=1}^k \left(D_i + (\Phi_L(\Gamma_i) \setminus \{0\})\right) \cup \Phi(\mathbb{L}(\Gamma)),
\end{equation}
where
\[
D_i = 
\begin{cases}
\sum_{u_j \sim u_i} n_j & \text{if } N_\Gamma(u_i) \neq \emptyset, \\
0 & \text{otherwise},
\end{cases}
\]
and
\begin{equation}
\mathbb{L}(\Gamma) = 
\begin{bmatrix}
D_1 & -p_{1,2} & \cdots & -p_{1,k} \\
-p_{2,1} & D_2 & \cdots & -p_{2,k} \\
\vdots & \vdots & \ddots & \vdots \\
-p_{k,1} & -p_{k,2} & \cdots & D_k
\end{bmatrix}
\end{equation}
with
\[
p_{i,j} = 
\begin{cases}
\sqrt{n_i n_j} & \text{if } u_i \sim u_j \text{ in } \Gamma, \\
0 & \text{otherwise}.
\end{cases}
\]
In equation (1), the notation $\Phi_L(\Gamma_i) \setminus \{0\}$ indicates the removal of one copy of the eigenvalue $0$ from the multiset $\Phi_L(\Gamma_i)$, and $D_i + (\Phi_L(\Gamma_i) \setminus \{0\})$ denotes the addition of $D_i$ to each element of $(\Phi_L(\Gamma_i) \setminus \{0\})$.
\end{thrm}

Consider the weighted graph obtained by assigning weight $n_i = |V(\Gamma_i)|$ to each vertex $u_i$ of $\Gamma$ for $i = 1, 2, \dotsc, k$. Define the $k \times k$ matrix $L(\Gamma) = (l_{i,j})$ by
\[
l_{i,j} = 
\begin{cases}
-n_j & \text{if } i \neq j \text{ and } u_i \sim u_j, \\
\sum_{u_i \sim u_r} n_r & \text{if } i = j, \\
0 & \text{otherwise}.
\end{cases}
\]

This matrix $L(\Gamma)$ is called the {vertex weighted Laplacian matrix} of $\Gamma$. 
Although $L(\Gamma)$ has zero row sums, it is generally not symmetric. In contrast, 
the matrix $\mathbb{L}(\Gamma)$ defined in Theorem 2.1 is symmetric but may not have 
zero row sums. Let $W$ be the $k \times k$ diagonal matrix with diagonal entries $n_1, n_2, \dotsc, n_k$.  Then as shown in \cite{chung1996}, $L(\Gamma) = W^{-\frac{1}{2}} \mathbb{L}(\Gamma) W^{\frac{1}{2}}$, which implies that $L(\Gamma)$ and $\mathbb{L}(\Gamma)$ are similar matrices. 
Hence, we obtain the following result.

\begin{prop}
    $\Phi_L(\Gamma) = \Phi(\mathbb{L}(\Gamma))$
\end{prop}

\section{Structure of cozero-divisor graph $\Gamma'(\Z_{n}[x]/(x^2))$}
Let $R = \mathbb{Z}_n[x]/(x^2)$, and note that its elements are of the form $ax + b$ with $a, b \in \Z_n$. $R^\times$ denotes the group of units of the ring $R$. We note that if $b\in\Z_n^\times$, then the element $ax + b\in R^\times$. Consequently, such elements cannot be vertices of the cozero-divisor graph of $R$. Now for any positive integer $n=p_1^{k_1}p_2^{k_2} \cdots p_r^{k_r}$, we obtain the following ring isomorphism $\Z_n\cong \Z_{p_1^{k_1}} \times \Z_{p_2^{k_2}} \times  \cdots \times \Z_{p_r^{k_r}}$. This can be extended to the polynomial (quotient) rings as well, thus $$\Z_n[x]/(x^2)\cong \Z_{p_1^{k_1}}[x]/(x^2) \times \Z_{p_2^{k_2}}[x]/(x^2)\times  \cdots \times \Z_{p_r^{k_r}}[x]/(x^2)$$
In this paper, we classify the structure and Laplacian spectrum of $\Gamma'(\Z_{n}[x]/(x^2))$ for key families of rings parameterized by $n$. Our analysis covers the cases $n=p,p^2,p^3,pq,p^2q$ and $pqr$, where $p,q,r$ are distinct primes. For each family, we decompose the ring into local components, determine the graph's structure by analyzing its ideals, and then compute its complete Laplacian spectrum, revealing spectral properties that depend directly on the prime factorization of $n$.

\subsection{Prime Case: $n = p$}

Note that $\Z_{p}[x]/(x^2)$ is a local ring with maximal ideal generated by $x$. Now the vertex set of the cozero-divisor graph $\Gamma'(\Z_{p}[x]/(x^2))$ is $$V(\Gamma'(\Z_{p}[x]/(x^2)))=\{ax:a\in \Z_p^\times\}$$
It can be easily observed that the ideal generated by \(ax\), for all $a\in \Z_p^\times$, is the same as the maximal ideal generated by $x$, thus $bx\in a\Z_{p}[x]/(x^2)$ for any $a,b\in \Z_p^\times$. Hence, we conclude that 
\begin{thrm}
$\Gamma'(\Z_{p}[x]/(x^2)) \cong \bar K_{p-1}$.
\end{thrm}

We denote the eigenvalues $\lambda_i$ of $\mathcal L (\Gamma'(\Z_n[x]/(x^2)))$ with multiplicuty $\mu_i$ by $\lambda_i^{[\mu_i]}$. Now the following theorem is obvious: 
\begin{thrm}
The Laplacian spectrum of $\Gamma'(\Z_{p}[x]/(x^2))$ is $0^{[p-1]}$.
\end{thrm} 
\subsection{Two distinct primes: $n = pq$}
For the ring $\Z_{pq}[x]/(x^2)$, where $p,q$ are both primes, we have $\Z_{pq}[x]/(x^2)\cong \Z_{p}[x]/(x^2) \times \Z_{q}[x]/(x^2) $. Let $S=\{c\in \Z_{pq}:(c,pq)\ne 1\}$. Then the vertex set of $\Gamma'(R)$ (where $R=\Z_{pq}[x]/(x^2)$): 

$$V(\Gamma'(R))=\{ax+b: a\in \Z_{pq},b\in S\}$$
It can be noted that $\mid V(\Gamma'(R)) \mid =pq(p+q-1)-1$, as $\mid S\mid =pq-\phi(pq)=p+q-1$. Every principal ideal in $\Z_{pq}[x]/(X^2)$ corresponds uniquely to an ordered pair of principal ideals, one from $\Z_{p}[x]/(X^2)$, and the other one from $\Z_{q}[x]/(X^2)$. For the local rings $\Z_{p}[x]/(X^2)$ and $\Z_{q}[x]/(X^2)$
the principal ideals can be categorized as follows:

\[
\begin{aligned}
\text{In } \mathbb{Z}_{p}[x]/(x^2):\ & I_1 = \langle 0 \rangle,\quad
I_2 = \langle x \rangle,\quad
I_3 = \{\langle a x + b \rangle : b \in \mathbb{Z}_p^\times\},\\[4pt]
\text{In } \mathbb{Z}_{q}[x]/(x^2):\ &
J_1 = \langle 0 \rangle,\quad
J_2 = \langle x \rangle,\quad
J_3 = \{\langle c x + d \rangle : d \in \mathbb{Z}_q^\times\}.
\end{aligned}
\]
The corresponding cardinalities are:
\[
\mid I_1\mid  = \mid J_1\mid = 1,\quad
\mid I_2\mid  = p-1,\quad
\mid J_2\mid = q-1,\quad
\mid I_3\mid = p(p-1),\quad
\mid J_3\mid  = q(q-1).
\]

Let $a_p,b_p \in \mathbb{Z}_p$ and $a_q,b_q \in \Z_q$ denote the images of $a,b \in \mathbb{Z}_{pq}$ modulo $p$ and $q$, respectively. Each principal ideal $\langle a x + b \rangle$ in $\Z_{pq}[x]/(x^2)$ can then be analyzed via the pair of ideals it generates in the component rings. There are nine possible combinations of $(I_i, J_j)$, for $i,j\in \{0,1,2\}$, which we examine below.


\begin{itemize}
\item $I_1 \times J_1$:  
Here $a,b$ are both divisible by $p$ and $q$, yielding only the zero element, and hence excluded from the vertex set.

\item$I_1 \times J_2$:  
Both $a,b$ divisible by $p$, with $q \mid b$ and $q \nmid a$. This yields the ideal $\langle px \rangle.$

\item$I_2 \times J_1$:  
Both $a,b$ divisible by $q$, with $p \mid b$ and $p \nmid a$. This gives $\langle qx \rangle.$

\item $I_2 \times J_2$:  
$(a,p)=1$, and $(a,q)=1$ which makes $a$ a unit in $Z_{pq}$. On the other hand, $b=0$, which corresponds to all maximal ideals $\langle a x \rangle \cong \langle x \rangle.$

\item $I_1 \times J_3$:  
Both $a,b$ are divisible by $p$, with $b_q$ a unit. This gives ideals analogous to $\langle p \rangle.$

\item $I_3 \times J_1$:  
$b_p$ is a unit in $\Z_p$ and $a_q = b_q = 0$, yielding ideals similar to $\langle q \rangle.$

\item $I_2 \times J_3$:  
$a_p=1$, $b_p=0$, and $b_q$ a unit in $\Z_q$. This gives ideals similar to $\langle x + p \rangle.$

\item $I_3 \times J_2$:  
$b_p$ a unit in $\Z_p$, $a_q = b_q = 0$. This yields ideals similar to $\langle x + q \rangle.$

\item $I_3 \times J_3$:  
Both $b_p$ and $b_q$ are units, giving unit elements only; these are excluded.
\end{itemize}

\begin{exam}{\rm
Let us construct the cozero-divisor graph of $R=\mathbb{Z}_{10}[x]/(x^2)$ as an example of the $n=pq$ case with $p=2$, $q=5$. The vertex set of \(\Gamma'(\mathbb{Z}_{10}[x]/(x^2))\) is $\{ax + b : a \in \mathbb{Z}_{10}, \, b \in \{0,2,4,5,6,8\}\}\setminus \{(0)\}$. Note that \(|V(\Gamma'(R))| = 59\). The principal ideals in $\mathbb{Z}_2[x]/(x^2)$ and $\mathbb{Z}_5[x]/(x^2)$ are categorized as follows:
$I_1 = J_1 = \langle 0 \rangle$, $I_2 = \langle x \rangle$, $J_2 = \langle x \rangle$, $I_3 = \{\langle ax + b \rangle : b \in \mathbb{Z}_2^\times\}$, and $J_3 = \{\langle ax + b \rangle : b \in \mathbb{Z}_5^\times\}$. This leads us to the following (non-zero, non-unit) combinations:
\begin{align*}
&I_1 \times J_2~:~~ \langle 2x \rangle \cong \{\langle mx \rangle : (m,10) = 2\} \\
&I_2 \times J_1~: ~~\langle 5x \rangle \\
&I_2 \times J_2~:~~ \langle x \rangle \cong \{\langle mx \rangle : m \in \mathbb{Z}_{10}^\times\} \\
&I_1 \times J_3~:~~ \langle 2 \rangle \cong\{\langle mx + n \rangle : m \in 2\mathbb{Z}_5,\ n \in 2\mathbb{Z}_5^\times\} \\
&I_3 \times J_1~:~~ \langle 5 \rangle \cong \langle 5x + 5 \rangle \\
&I_2 \times J_3~:~~ \langle x + 2 \rangle \cong \{\langle mx + n \rangle : m \in \{1,3,5,7,9\},\ n \in \{2,4,6,8\}\} \\
&I_3 \times J_2~:~~ \langle x + 5 \rangle \cong \{\langle mx + 5 \rangle : (m,10) \ne 5\}
\end{align*}
This gives all $59$ vertices of $\Gamma'(\mathbb{Z}_{10}[x]/(x^2))$.}
\end{exam}

Now we define the following sets: 
$${\cal A}_{i,j}=\{cx+d\in \Z_{pq}[x]/(x^2): \langle cx+d \rangle\in I_i\times J_j\}$$
where each set is a collection of ideals of the form $I_i\times J_j$.

\begin{thrm}\label{pqreducedgraphs}
The induced subgraphs of $\Gamma'(R)$ on the subsets $\A_{i,j}$ for all $(i,j)\in \{1,2,3\}$ are as follows:
\begin{align*}
& \Gamma'(\A_{1,2}) \cong \bar K_{q-1},\quad
 \Gamma'(\A_{2,1}) \cong \bar K_{p-1},\quad
 \Gamma'(\A_{1,3})\cong\bar K_{q(q-1)},\quad
\Gamma'(\A_{3,1})\cong\bar K_{p(p-1)}\\
& \Gamma'(\A_{2,2})\cong \bar K_{(p-1)(q-1)}\quad
 \Gamma'(\A_{2,3})\cong \bar K_{q(p-1)(q-1)}\quad
\Gamma'(\A_{3,2})\cong \bar K_{p(p-1)(q-1)}\quad
\end{align*}
\end{thrm}

The total number of ideals listed above is $pq(p+q-1)-1$, which confirms that the enumeration is exhaustive. This leads to the following structural result.
\begin{lem}\label{2plemma}
The sets $\A_{i,j}$ for all $i,j\in \{1,2,3\}$ are pairwise disjoint, and we can partition the vertex set of $\Gamma'(R)$ as 
$ V(\Gamma'(R))=\bigcup_{(i,j)\in \{1,2,3\}} \A_{i,j}$. 
\end{lem}

We now define the reduced (or quotient) graph $\Upsilon'(\Z_{pq}[x]/(x^2))$ as a simple, undirected graph. Its vertices are the sets $\mathcal{A}_{i,j}$ for all $(i,j) \in \{1,2,3\}$. Next, we examine the adjacency conditions in $\Upsilon'(\Z_{pq}[x]/(x^2))$. In the rings $R = \Z_p$ or $\Z_q$, we have the chain of ideals $\langle 0\rangle \subseteq \langle x\rangle \subseteq R$. Consequently, two maximal ideals $I_i \times J_j$ and $I_k \times J_l$ in $\Z_{pq}[x]/(x^2)$ are adjacent if and only if $I_i \not\subset I_k$, $I_k \not\subset I_i$, $J_j \not\subset J_l$, and $J_l \not\subset J_j$. This leads to the cluster graph of the cozero-divisor graph shown in Figure \ref{2pcozerograph}. An important property of this graph is that if any vertex from $\A_{i,j}$ is adjacent to a vertex from $\A_{k,l}$, then every vertex in $\A_{i,j}$ is adjacent to every vertex in $\A_{k,l}$.

\begin{figure}[h!]
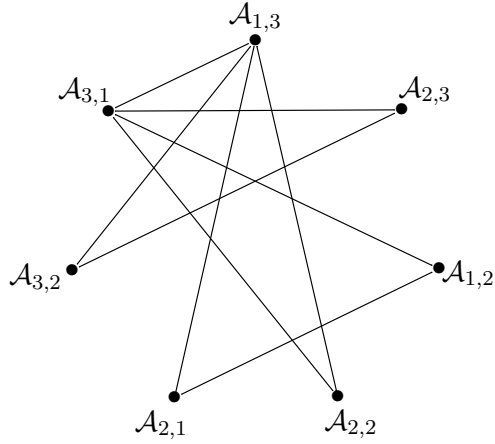

\[
\xygraph{
!{<0cm,0cm>;<0cm,1cm>:<-1cm,0cm>::}
!{(0,0);a(0)**{}?(2.5)}*{\bullet}="v0" !{(0,0);a(0)**{}?(2.8)}*{\A_{1,3}}
!{(0,0);a(51.5)**{}?(2.5)}*{\bullet}="v1" !{(0,0);a(51.5)**{}?(2.9)}*{\A_{3,1}}
!{(0,0);a(103)**{}?(2.5)}*{\bullet}="v2" !{(0,0);a(103)**{}?(3)}*{\A_{3,2}}
!{(0,0);a(154.5)**{}?(2.5)}*{\bullet}="v3" !{(0,0);a(154.5)**{}?(2.9)}*{\A_{2,1}}
!{(0,0);a(206)**{}?(2.5)}*{\bullet}="v4" !{(0,0);a(206)**{}?(2.9)}*{\A_{2,2}}
!{(0,0);a(257.5)**{}?(2.5)}*{\bullet}="v5" !{(0,0);a(257.5)**{}?(2.9)}*{\A_{1,2}}
!{(0,0);a(309)**{}?(2.5)}*{\bullet}="v6" !{(0,0);a(309)**{}?(2.9)}*{\A_{2,3}}
"v0"-"v1" "v0"-"v2" "v0"-"v3" "v0"-"v4" 
"v1"-"v4" "v1"-"v5" "v1"-"v6"
"v3"-"v5" "v6"-"v2"  
}
\]
\caption{The reduced graph $\Upsilon'[\Z_{pq}[x]/(x^2)]$.}
\label{2pcozerograph}
\end{figure}

\begin{thrm}
The Laplacian spectrum of $\Gamma'(\Z_{pq}[x]/(x^2))$ is $\{pq(p-1)^{[q(p-1)(q-1)-1]}, \\pq(q-1)^{[p(p-1)(q-1)-1]} \}$ and the remaining $7$ eigenvalues are the eigenvalues of the matrix provided below.
\end{thrm}
\begin{proof}
First note that \[\Gamma'(\Z_{pq}[x]/(x^2))=\Upsilon'[\Gamma'(\A_{1,2}),\Gamma'(\A_{2,1}),\Gamma'(\A_{2,2}),\Gamma'(\A_{3,1}),\Gamma'(\A_{3,1}),\Gamma'(\A_{3,2}),\Gamma'(\A_{2,3})]\]
From Theorem \ref{pqreducedgraphs}, we have 
\begin{align*}
D_{1,2} = p^2 - 1, \quad &
D_{2,1} = q^2 - 1, \\
D_{1,3} = q(p-1)^2, \quad & D_{3,1} = p(q-1)^2, \\
D_{2,2} = p(p-1) + q(q-1), &\quad D_{2,3} = pq(p-1), \quad D_{3,2} = pq(q-1).
\end{align*}

Therefore 
\begin{align*}
\Phi_{L}(\Gamma'(\Z_{pq}[x]/(x^2))) &= \bigcup_{i,j\in \{1,2,3\}, 3\le i+j \le 5}\left(D_{i,j} + (\Phi_L(\Gamma'(\A_{i,j})) \setminus \{0\})\right) \bigcup  \Phi(\mathbb{L}(\Upsilon'[\Z_{pq}[x]/(x^2)])) \\ 
&= \begin{pmatrix}
p^2-1 & q^2-1 & p(p-1)+q(q-1) & q(p-1)^2 &p(q-1)^2 \\
q-2   & p-2   & (p-1)(q-1)-1  & q(q-1)-1 &p(p-1)-1 
\end{pmatrix}\\
&\bigcup  \begin{pmatrix}
pq(p-1) & pq(q-1)\\
q(p-1)(q-1)-1 & p(p-1)(q-1)-1
\end{pmatrix}\bigcup \Phi(\mathbb{L}(\Upsilon'[\Z_{pq}[x]/(x^2)])).
\end{align*}
Where the remaining eigenvalues are the eigenvalue of the matrix $\Phi(\mathbb{L}(\Upsilon'[\Z_{pq}[x]/(x^2)]))=$ 
\[
\begin{pmatrix}
p^2 - 1 & -(p-1) & 0 & 0 & -p(p-1) & 0 & 0 \\
-(q-1) & q^2 - 1 & 0 & -q(q-1) & 0 & 0 & 0 \\
0 & 0 & q(q-1) + p(p-1) & -q(q-1) & -p(p-1) & 0 & 0 \\
0 & -(p-1) & -(p-1)(q-1) & q(p^2 - 1) & -p(p-1) & 0 & -p(p-1)(q-1) \\
-(q-1) & 0 & -(p-1)(q-1) & -q(q-1) & p(q^2 - 1) & -q(p-1)(q-1) & 0 \\
0 & 0 & 0 & 0 & -p(p-1) & p q (p-1) & -p(p-1)(q-1) \\
0 & 0 & 0 & -q(q-1) & 0 & -q(p-1)(q-1) & p q (q-1)
\end{pmatrix}
\]
\end{proof}

\subsection{Prime square case: $n = p^2$}

Let $p$ be an odd prime. Note that \[V(\Gamma'(\Z_{p^2}[x]/(x^2)))=\{ax+b: a\in \Z_{p^2}, b\not\in \Z_{p^2}^\times\}\setminus\{0\} \] A straightforward count shows that the order of this graph is $(p^2)\cdot(p^2-\phi(p^2))-1=p^3-1$. We now classify the vertices by analyzing the possible forms of $ax + b$, considering the divisibility of $a$ and $b$ by $p$.

\begin{mycases}
    \case In this case we consider $b = 0$, which leads to two subcases:
    \begin{itemize}
        \item[(i)] If $(a, p^2) = 1$, then the element corresponds to the maximal ideal $\langle x \rangle$. There are $\varphi(p^2) = p(p - 1)$ such elements.
        
        \item[(ii)] If $p \mid a$, the corresponding ideal is $\langle px \rangle$. Excluding the zero element $(a,b) = (0,0)$, there are $p - 1$ such elements. 
    \end{itemize}

    \case Next we consider $b>0$. As $p \mid b$ we consider the following three subcases:
    \begin{itemize}
        \item[(i)] If $a = 0$, the corresponding maximal ideal is $\langle p \rangle$, and there are $\varphi(p^2) = p - 1$ such elements.
        
        \item[(ii)] If $a$ is a unit in $\mathbb{Z}_{p^2}$, we obtain ideals of the form 
        \[
        \langle ax + lp \rangle, \quad \text{where } (a,p^2)=1 \text{ and } l \in \mathbb{Z}_p^\times.
        \]
        For each $l \in \mathbb{Z}_p^\times$, there are $\varphi(p^2) = p(p - 1)$ such elements, giving rise to $p - 1$ distinct ideals.
        
        \item[(iii)] If $p \mid a$, the elements are of the form $\langle kpx + lp \rangle$, where both $k, l \in \mathbb{Z}_p^\times$ are units. There are $(p - 1)^2$ such elements. However, since 
        \[
        \langle p(kx + l) \rangle = \langle p \rangle,
        \]
        these correspond to the same maximal ideal $\langle p \rangle$ already considered. Therefore, the total number of maximal ideals of this type is 
        \[
        (p - 1) + (p - 1)^2 = p(p - 1).
        \]
    \end{itemize}
\end{mycases}

In summary, the nonzero and non-unit elements of $\mathbb{Z}_{p^2}[x]/(x^2)$ are distributed among the following principal ideals, which constitute the vertex set of the cozero-divisor graph $\Gamma'(\mathbb{Z}_{p^2}[x]/(x^2))$:

\begin{itemize}
    \item $\langle x \rangle$, containing $p(p - 1)$ elements. We denote by $\A_{1,0}$ the set of all elements generating this maximal ideal.

    \item $\langle px \rangle$, containing $p - 1$ elements. We denote by $\A_{p,0}$ the set of all elements generating this ideal.

    \item $\langle p \rangle$, containing $p(p - 1)$ elements (combining Subcases~2(i) and~2(iii)). We denote by $\A_{0,p}$ the set of all elements generating this maximal ideal.

    \item A family of $p - 1$ distinct ideals of the form $\langle ax + \ell p \rangle$ (arising from Subcase~2(ii)), each containing $p(p - 1)$ elements. For each $\ell \in \Z_p^\times$, we denote by $\A_{1, \ell p}$ the set of all elements generating the corresponding maximal ideal.
\end{itemize}
The total number of ideals listed above is $p^3-1$, which confirms that the enumeration is exhaustive. Among these, it is readily seen that the ideal $\langle px \rangle$ is contained in every other maximal ideal. On the other hand, the ideals $\langle x\rangle$, $\langle p\rangle$ are incomparable.

Next, consider two distinct ideals from the class $\A_{1,\ell p}$, generated by $y = x + kp$ and  $z = x + lp$, where $k, l \in \mathbb{Z}_{p}^\times$ and $k \not\equiv l \pmod{p}$. Suppose, for contradiction, that  $\langle x + kp \rangle \subseteq \langle x + lp \rangle$. Then there exist $c, d \in \Z_{p^2}$, such that
\[
x + kp = (c + dx)(x + lp)
\]
Expanding and comparing coefficients gives the following: $kp = clp$, and  $1 = c + dlp$.
From the first relation, $k-cl$ must be a multiple of $p$. Substituting $c = 1-dlp$ into this expression gives $k-(1-dlp)l=k-l-dl^2p$, which must also be a multiple of $p$. Hence $k - l$ is divisible by $p$, contradicting our assumption that $k \not\equiv l \pmod{p}$. Therefore, $\langle x + kp \rangle \not\subseteq \langle x + lp\rangle$ for any $k,l\in \Z_p^\times$. By symmetry, $\langle x + lp \rangle \not\subseteq \langle x + kp \rangle$ either. Thus, all ideals in $\A_{1,\ell p}$ are pairwise disjoint, and consequently all vertices in $\A_{1,\ell p}$ are adjacent to each other. 

Similarly, it can be shown that $\langle x\rangle$, $\langle p\rangle$ are not contained within $\langle x + \ell p\rangle$, for any $\ell \in \Z_p^\times$, and vice-versa. Consequently, in the reduced (quotient) graph $\Upsilon'[\Z_{p^2}[x]/(x^2)]$, the vertex corresponding to $\A_{p,0}$ is isolated, while all remaining vertices form a complete subgraph (clique). The resulting structure is illustrated in Figure \ref{p^2cozero}.

\begin{figure}[h!]
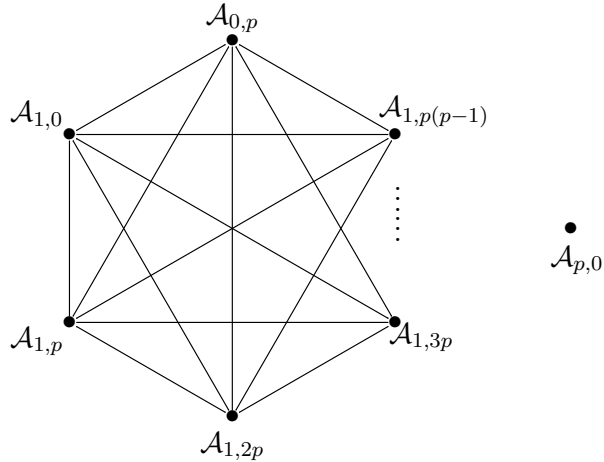

\[
\xygraph{
!{<0cm,0cm>;<0cm,1cm>:<-1cm,0cm>::}
!{(0,0);a(0)**{}?(2.5)}*{\bullet}="v0" !{(0,0);a(0)**{}?(2.8)}*{\A_{0,p}}
!{(0,0);a(60)**{}?(2.5)}*{\bullet}="v1" !{(0,0);a(60)**{}?(3)}*{\A_{1,0}}
!{(0,0);a(120)**{}?(2.5)}*{\bullet}="v2" !{(0,0);a(120)**{}?(3)}*{\A_{1,p}}
!{(0,0);a(180)**{}?(2.5)}*{\bullet}="v3" !{(0,0);a(180)**{}?(2.9)}*{\A_{1,2p}}
!{(0,0);a(240)**{}?(2.5)}*{\bullet}="v4" !{(0,0);a(240)**{}?(2.9)}*{\A_{1,3p}}
!{(0,0);a(300)**{}?(2.5)}*{\bullet}="v5" !{(0,0);a(300)**{}?(3.1)}*{\A_{1,p(p-1)}}
!{(0.5,-2.2)}*{\vdots}="vdots" !{(0.1,-2.2)}*{\vdots}="vdots"
!{(0,0);a(270)**{}?(4.5)}*{\bullet}="v" !{(0,.06);a(265)**{}?(4.4)}*{\A_{p,0}}
"v0"-"v1" "v1"-"v2" "v2"-"v3" "v3"-"v4" 
"v0"-"v2" "v0"-"v3" "v0"-"v4" "v0"-"v5" 
"v1"-"v3" "v1"-"v4" "v1"-"v5" "v2"-"v4" "v2"-"v5" "v3"-"v5"
}
\]
\caption{The reduced graph $\Upsilon'[\Z_{p^2}[x]/(x^2)]$}
\label{p^2cozero}
\end{figure}

\begin{thrm}
The Laplacian spectrum of $\Gamma'(\Z_{p^2}[x]/(x^2))$ is $\big\{0^{[p]},p(p^2-1) ^{[p^3-p-1]}\big\}$.
\end{thrm}
\begin{proof}
First note that \[\Gamma'(\Z_{p^2}[x]/(x^2))=\Upsilon'[\Gamma'(\A_{0,p}),\Gamma'(\A_{p,0}),\bigcup_{\ell \in\Z_p} \Gamma'(\A_{1,\ell p})]\]
From Theorem \ref{pqreducedgraphs}, for all $\ell \in \Z_p$, we have 
$n_{p,0} = 0, n_{0,p} = n_{1,\ell p} = p(p - 1)$, and consequently, $ D_{p,0} = 0, 
D_{0,p} = D_{1,\ell p} = p(p^2 - 1)$. Therefore, the Laplacian spectrum is given by:
\begin{align*}
\Phi_{L}(\Gamma'(\Z_{p^2}[x]/(x^2))) &= \left(D_{p,0} + (\Phi_L(\Gamma'(\A_{p,0})) \setminus \{0\})\right) \bigcup\left(D_{0,p} + (\Phi_L(\Gamma'(\A_{0,p})) \setminus \{0\})\right) \\&\bigcup_{\ell\in \Z_p} \left(D_{1,\ell p} + (\Phi_L(\Gamma'(\A_{1,\ell p})) \setminus \{0\})\right)\bigcup  \Phi(\mathbb{L}(\Upsilon'[\Z_{p^2}[x]/(x^2)])) \\ 
&= \begin{pmatrix}
0 & p(p^2-1)   \\
p-2 & (p+1)(p(p-1)-1) 
\end{pmatrix}  \bigcup \Phi(\mathbb{L}(\Upsilon'[\Z_{p^2}[x]/(x^2)])).
\end{align*}
The remaining eigenvalues are the eigenvalue of the matrix $\Phi(\mathbb{L}(\Upsilon'[\Z_{p^2}[x]/(x^2)]))$, whose structure corresponds to a graph with two components: a complete graph $K_{p+1}$ where each vertex has weight $p(p-1)$, and an isolated vertex with weight $0$. It is known that the Laplacian spectrum of $K_n$ is $\{0^{[n-1]},n^{[1]} \}$. Thus, the combined Laplacian spectrum of this union of two-component graphs is $\{0^{[2]},(p+1)p(p-1)^{[p]} \}$. Combining all parts yields the claimed spectrum.
\end{proof}

\subsection{Three distinct primes: $n = pqr$}
For the ring $\Z_{pqr}[x]/(x^2)$, the vertex set of the cozero-divisor graph is 
\[V(\Gamma'(\Z_{pqr}[x]/(x^2)))=\{ax+b: a\in \Z_{pqr},b\in S\}\setminus\{0\}\]
where $S=\{c\in \Z_{pqr}:(c,pqr)\ne 1\}$. It can be noted that $\mid V(\Gamma'(R)) \mid =pqr(p+q+r-1)-1$, as $\mid S\mid =pqr-\phi(pqr)=p+q+r-1$. Every principal ideal in $\Z_{pqr}[x]/(X^2)$ corresponds uniquely to an ordered pair of principal ideals, one from each component ring $\Z_{p}[x]/(X^2)$, $\Z_{q}[x]/(X^2)$, and $\Z_{r}[x]/(X^2)$. The principal ideals can be categorized as follows:

\[
\begin{aligned}
\text{In } \mathbb{Z}_{p}[x]/(x^2):\ & 
I_1 = \langle 0 \rangle,\quad
I_2 = \langle x \rangle,\quad
I_3 = \{\langle a x + b \rangle : b \in \Z_p^\times\}=\Z_p[x]/(x^2),\\[4pt]
\text{In } \mathbb{Z}_{q}[x]/(x^2):\ & 
J_1 = \langle 0 \rangle,\quad
J_2 = \langle x \rangle,\quad
J_3 = \{\langle c x + d \rangle : d\in \Z_q^\times\}=\Z_q[x]/(x^2),\\[4pt]
\text{In } \mathbb{Z}_{r}[x]/(x^2):\ &
K_1 = \langle 0 \rangle,\quad
K_2 = \langle x \rangle,\quad
K_3 = \{\langle e x + f \rangle : f \in \Z_r^\times\}=\Z_r[x]/(x^2).
\end{aligned}
\]
The corresponding cardinalities are:
\[
\mid I_1\mid  = \mid J_1\mid = 1,
\mid I_2\mid  = p-1,
\mid J_2\mid = q-1,
\mid K_2\mid = r-1,
\mid I_3\mid = p(p-1),
\mid J_3\mid = q(q-1),
\mid K_3\mid = r(r-1).
\]

Now we define the following sets: 
\[\A_{i,j,k}=\{ax+b\in \Z_{pqr}[x]/(x^2): \langle ax+b \rangle\in I_i\times J_j\times K_k\}\]
for $i,j,k\in \{1,2,3\}$ such that $4\le i+j+k \le 11$ ($i+j+k=3$ and $12$ are excluded as they correspond to the zero ideal and the entire ring $\Z_{pqr}[x]/(x^2)$). The complete list of all $25$ possible combinations of the sets $\A_{i,j,k}$ is presented in Table~\ref{25combination_D_full}. Next, we analyze these sets to determine the adjacency relations in the cozero-divisor graph $\Gamma'(\Z_{pqr}[x]/(x^2))$. To facilitate this analysis, we first state the following theorem.

\begin{thrm}\label{connectingtheorempqr}
In the cozero-divisor graph $\Gamma'(\Z_{pqr}[x]/(x^2))$, every vertex belonging to the set $\A_{x_1,x_2,x_3}$ is adjacent to every vertex in the set $\A_{y_1,y_2,y_3}$ whenever there exist indices $i,j\in \{1,2,3\}$ such that $x_i<y_i$ and $x_j>y_j$.
\end{thrm}

\begin{proof}
Each principal ideal of $\Z_{pqr}[x]/(x^2)$ corresponds to an ordered triple of component ideals from $\Z_p[x]/(x^2)$, $\Z_q[x]/(x^2)$, and $\Z_r[x]/(x^2)$. In each local component ring $R=\Z_t[x]/(x^2)$ (for $t\in\{p,q,r\}$), the lattice of ideals is given by $\langle 0 \rangle \subset \langle x \rangle \subset R$.
Hence, every principal ideal of $\Z_{pqr}[x]/(x^2)$ can be represented as a product $I_{i}\times J_{j}\times K_{k}$, where $I_i$, $J_j$, and $K_k$ are ideals of $\Z_p[x]/(x^2)$, $\Z_q[x]/(x^2)$, and $\Z_r[x]/(x^2)$, respectively. Accordingly, we have
\[
\A_{x_1,x_2,x_3}=\{\, ax+b\in \Z_{pqr}[x]/(x^2) : 
\langle a_px+b_p\rangle = I_{x_1},\,
\langle a_qx+b_q\rangle = J_{x_2},\,
\langle a_rx+b_r\rangle = K_{x_3}\,\},
\]
and similarly for $\A_{y_1,y_2,y_3}$, where $x_i, y_i \in \{1,2,3\}$ denote the positions in the corresponding ideal chains.

Let $u\in \A_{x_1,x_2,x_3}$ and $v\in \A_{y_1,y_2,y_3}$, and denote their generated principal ideals by
\[
\langle u\rangle = P = I_{x_1}\times J_{x_2}\times K_{x_3}, \quad
\langle v\rangle = Q = I_{y_1}\times J_{y_2}\times K_{y_3}.
\]
Two vertices $u$ and $v$ are adjacent in $\Gamma'(\Z_{pqr}[x]/(x^2))$ if and only if their corresponding ideals $P$ and $Q$ are incomparable, that is, $P\nsubseteq Q \quad \text{and} \quad Q\nsubseteq P$. 
Since the inclusion of product ideals is determined coordinate-wise, we have $P\subseteq Q$ if and only if $I_{x_k}\subseteq I_{y_k} \text{ for all } k\in\{1,2,3\}$. Therefore, if there exist indices $i,j\in \{1,2,3\}$ such that $x_i<y_i$ and $x_j>y_j$, it follows immediately that $P\nsubseteq Q$ and $Q\nsubseteq P$. Hence, $u$ and $v$ are adjacent. As $u$ and $v$ were arbitrary elements of $\A_{x_1,x_2,x_3}$ and $\A_{y_1,y_2,y_3}$, it follows that every vertex in $\A_{x_1,x_2,x_3}$ is adjacent to every vertex in $\A_{y_1,y_2,y_3}$. This completes the proof.
\end{proof}

As a direct consequence of the preceding theorem, the degree of a vertex in the reduced graph $\Upsilon'[\Z_{pqr}[x]/(x^2)]$ corresponding to an ideal of type $\A_{i,j,k}$ is given by $28-ijk-(4-i)(4-j)(4-k)$.


\begin{thrm}
The Laplacian spectrum of the cozero-divisor graph $\Gamma'(\Z_{pqr}[x]/(x^2))$ decomposes as follows:
\[
\Phi_{L}(\Gamma'(\Z_{pqr}[x]/(x^2))) = \bigcup_{\substack{(i,j,k) \in S}} D_{i,j,k}^{[|\A_{i,j,k}| - 1]} \bigcup \Phi_{L}(\mathbb{L}(\Upsilon'[\Z_{pqr}[x]/(x^2)]))
\]
where $S=\{1,2,3\}^3\setminus \{(1,1,1),(3,3,3)\}$, $D_{i,j,k}$ are the degrees of each vertex in $\A_{i,j,k}$ (see Table \ref{25combination_D_full}, and $\mathbb{L}(\Upsilon'[\Z_{pqr}[x]/(x^2)])$ is the $25 \times 25$ quotient matrix of the reduced graph in Figure~\ref{figcozeropqr}.
\end{thrm}

\begin{table}[h!]
  \centering
  {
\begin{tabular}{|c|c|c|c|}
      \hline
      {\textbf{Set}} & {\textbf{Representative}} & {\textbf{Cardinality}} & {\textbf{$D_{i,j,k}$}} \\
      \hline
      \(\A_{1,1,2}\) & \(\langle pqx\rangle\) & $r-1$ & $p^2 q^2 - 1$ \\
      \hline
      \(\A_{1,1,3}\) & \(\langle pq\rangle\) & $r(r-1)$ & $r(p^2 q^2 -1)$ \\
      \hline
      \(\A_{1,2,1}\) & \(\langle prx\rangle\) & $q-1$ & $p^2 r^2 -1$ \\
      \hline
      \(\A_{1,2,2}\) & \(\langle px\rangle\) & $(q-1)(r-1)$ & $(p^2-1)(q^2 + r - 1)$ \\
      \hline
      \(\A_{1,2,3}\) & \(\langle px+pq\rangle\) & $r(q-1)(r-1)$ & $r(p^2-1)(q^2 + r - 1)$ \\
      \hline
      \(\A_{1,3,1}\) & \(\langle pr\rangle\) & $q(q-1)$ & $q(p^2 r^2 -1)$ \\
      \hline
      \(\A_{1,3,2}\) & \(\langle px+pr\rangle\) & $q(q-1)(r-1)$ & $q(p^2-1)(q^2 + r - 1)$ \\
      \hline
      \(\A_{1,3,3}\) & \(\langle p\rangle\) & $qr(q-1)(r-1)$ & $qr(p^2-1)(q^2 + r - 1)$ \\
      \hline
      \(\A_{2,1,1}\) & \(\langle qrx\rangle\) & $p-1$ & $q^2 r^2 -1$ \\
      \hline
      \(\A_{2,1,2}\) & \(\langle qx\rangle\) & $(p-1)(r-1)$ & $(q^2-1)(p^2 + r - 1)$ \\
      \hline
      \(\A_{2,1,3}\) & \(\langle qx+pq\rangle\) & $r(p-1)(r-1)$ & $r(q^2-1)(p^2 + r - 1)$ \\
      \hline
      \(\A_{2,2,1}\) & \(\langle rx\rangle\) & $(p-1)(q-1)$ & $p^2 q^2 -1$ \\
      \hline
      \(\A_{2,2,2}\) & \(\langle x\rangle\) & $(p-1)(q-1)(r-1)$ & $p^2q^2 + p^2r^2 + q^2r^2 - p^2 - q^2 - r^2 + 1 - 2pqr$ \\
      \hline
      \(\A_{2,2,3}\) & \(\langle x+pq\rangle\) & $r(p-1)(q-1)(r-1)$ & $r(p^2q^2 + p^2r^2 + q^2r^2 - p^2 - q^2 - r^2 + 1 - 2pqr)$ \\
      \hline
      \(\A_{2,3,1}\) & \(\langle rx+pr\rangle\) & $q(p-1)(q-1)$ & $q(p^2 q^2 -1)$ \\
      \hline
      \(\A_{2,3,2}\) & \(\langle x+pr\rangle\) & $q(p-1)(q-1)(r-1)$ & $q(p^2q^2 + p^2r^2 + q^2r^2 - p^2 - q^2 - r^2 + 1 - 2pqr)$ \\
      \hline
      \(\A_{2,3,3}\) & \(\langle x+p\rangle\) & $qr(p-1)(q-1)(r-1)$ & $p q r (p-1)(q + r - 1)$ \\
      \hline
      \(\A_{3,1,1}\) & \(\langle qr\rangle\) & $p(p-1)$ & $p(q^2 r^2 -1)$ \\
      \hline
      \(\A_{3,1,2}\) & \(\langle qx+qr\rangle\) & $p(p-1)(r-1)$ & $p(q^2-1)(p^2 + r - 1)$ \\
      \hline
      \(\A_{3,1,3}\) & \(\langle q\rangle\) & $pr(p-1)(r-1)$ & $pr(q^2-1)(p^2 + r - 1)$ \\
      \hline
      \(\A_{3,2,1}\) & \(\langle rx+qr\rangle\) & $p(p-1)(q-1)$ & $p(p^2 q^2 -1)$ \\
      \hline
      \(\A_{3,2,2}\) & \(\langle x+qr\rangle\) & $p(p-1)(q-1)(r-1)$ & $p(p^2q^2 + p^2r^2 + q^2r^2 - p^2 - q^2 - r^2 + 1 - 2pqr)$ \\
      \hline
      \(\A_{3,2,3}\) & \(\langle x+q\rangle\) & $pr(p-1)(q-1)(r-1)$ & $pr(p^2q^2 + p^2r^2 + q^2r^2 - p^2 - q^2 - r^2 + 1 - 2pqr)$ \\
      \hline
      \(\A_{3,3,1}\) & \(\langle r\rangle\) & $pq(p-1)(q-1)$ & $pq(p^2 q^2 -1)$ \\
      \hline
      \(\A_{3,3,2}\) & \(\langle x+r\rangle\) & $pq(p-1)(q-1)(r-1)$ & $p q r (p + q - 1)(r - 1)$ \\
      \hline
    \end{tabular}
  }
  \caption{$25$ combinations of $(I_i,J_j,K_k)$ with exact $D_{i,j,k}$ from the provided formulas.}  
  \label{25combination_D_full}
\end{table}

\begin{figure}[h!]
    \centering
  \includegraphics[width=0.83\linewidth]{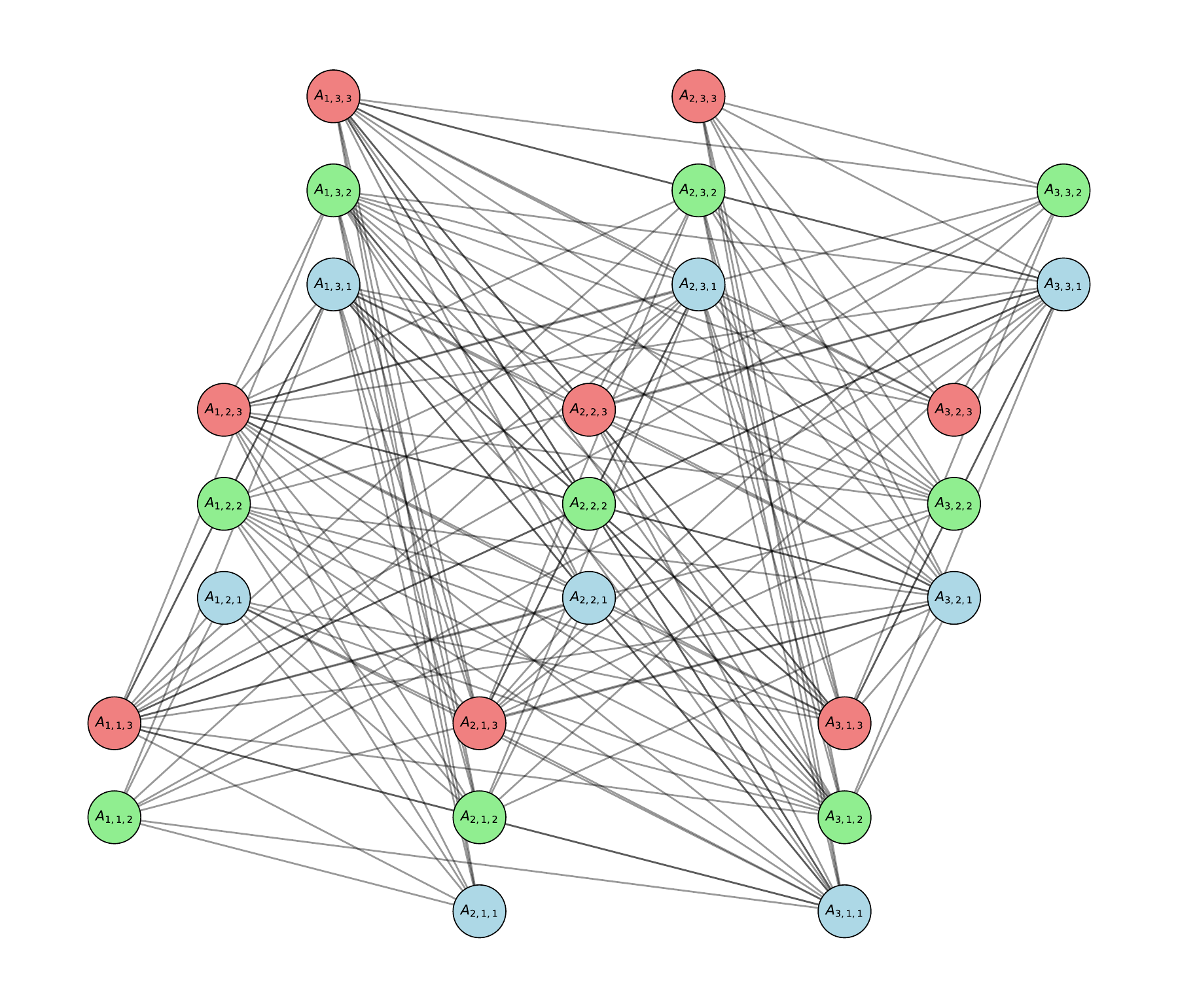}
    \caption{The reduced graph $\Upsilon'[\Z_{pqr}[x]/(x^2)]$}
    \label{figcozeropqr}
\end{figure}

\subsection{Cubic prime powers: $n = p^3$}

The graph $\Gamma'(\Z_{p^3}[x]/(x^2))$ has vertex set $\{ax+b : a\in \Z{p^3}, b\notin \Z_{p^3}^\times\}\setminus\{0\}$ with cardinality $(p^3)\cdot(p^3-\phi(p^3))-1=p^5-1$. First, we prove the following lemma. 

\begin{lem}\label{lemmacongideal}
For any $0\le k<l\le 2$, and $a_1,a_2, b_1,b_2$ are units in $\Z_{p^3}$, two ideals $\langle p^k a_1 x+p^l b_1 \rangle$ and $\langle p^k a_2 x+p^l b_2 \rangle$are identical if and only if  $a_2b_1\equiv a_1b_2\pmod {p^{l-k}}$.
\end{lem}
\begin{proof}
The two ideals are identical if there exists an element $cx+d\in\Z_{p^3}[x]/(x^2)$ such that $(p^k a_1 x+p^l b_1)(cx+d)=(p^k a_2 x+p^l b_2)$ for some $c,d\in\Z_{p^3}$. Comparing coefficients, we have the following congruences: $p^l b_1 d \equiv p^l b_2 \pmod {p^3}$ and $p^k a_1 d +p^l b_1 c\equiv p^k a_2 \pmod {p^3}$. From the first congruence, we get $b_1d \equiv  b_2 \pmod {p^{3-l}}$, which implies $d \equiv  b_2/b_1 \pmod {p^{3-l}}$. 

On the other hand, the second congruence gives us $p^k (a_2- a_1d) \equiv p^l b_1c \pmod {p^3}$, which simplifies $a_2- a_1d \equiv 0 \pmod {p^{l-k}}$. Finally, substituting $d$, we get the desired congruence $a_2b_1\equiv a_1b_2$ $\pmod{p^{l-k}}$.
\end{proof}

To obtain the structure of the cozero-divisor graph of ${\Z}_{p^3}[x]/(x^2)$, we first need to identify all distinct principal ideals. For an ideal element  $\langle ax+b:p\mid b\rangle$, we consider the following cases. 

\begin{mycases}
    \case $b = 0$.  In this case, we have two subcases:
    \begin{itemize}
        \item[(i)] If $(a, p^3) = 1$, then the element corresponds to the principal ideal $\langle x \rangle$. There are $\varphi(p^3) = p^2(p - 1)$ such ideals.
        
        \item[(ii)] If $p^2 \mid a$, the corresponding ideal is $\langle p^2x \rangle$. Excluding the zero element $(a,b) = (0,0)$, there are $p - 1$ such ideals. 

        \item[(iii)] $p \mid a$, but $p^2\nmid a$, the corresponding ideal is $\langle px \rangle$. Excluding the zero element $(a,b) = (0,0)$, there are $p(p-1)$ such ideals. 

Now, it can be easily observed that not only are these principal ideals distinct, but also that we can organize them as follows:
 $$\langle p^2x \rangle \subset \langle px \rangle \subset \langle p \rangle$$
    \end{itemize}

    \case For $b>0$ such that $p^2 \mid b$, which will be the following subcases:
    \begin{itemize}
        \item[(i)] If $a = 0$, the corresponding maximal ideal is $\langle p^2 \rangle$, and there are $\varphi(p^2) =p(p-1)$ such ideals.
        
        \item[(ii)] If $a$ is a unit in $\mathbb{Z}_{p^3}$, we obtain ideals of the form 
        \[
        \langle ax + lp^2 \rangle, \quad \text{where } (a,p^3)=1 \text{ and } l \in \mathbb{Z}_p^\times.
        \]
        For each $l \in \mathbb{Z}_p^\times$, there are $\varphi(p^3) = p^2(p - 1)$ such elements. Observe that by Lemma \ref{lemmacongideal}, two ideals $\langle x + l_1p^2 \rangle$ and $\langle x + l_2p^2 \rangle$ (without loss of generality, we consider $a=1$) are identical if and only if $l_1\equiv l_2 \pmod p$. Since this congruence holds only when $l_1 = l_2$, we conclude that there are exactly $p-1$ distinct ideals of this form.
        
        \item[(iii)] If $p^2 \mid a$, the elements are of the form $\langle kp^2x + lp^2 \rangle$, where both $k, l \in \mathbb{Z}_p^\times$ are units. There are $(p - 1)^2$ such ideals. However, since 
        \[
        \langle p^2(kx + l) \rangle = \langle p^2 \rangle,
        \]
        these correspond to the same maximal ideal $\langle p^2 \rangle$ already considered. Therefore, the total number of maximal ideals of this type is 
        \[
        (p - 1) + (p - 1)^2 = p(p - 1).
        \]      
        \item[(iv)]  Finally, suppose $p\mid a$, but $p^2\nmid a$. Then $a = kp$, where $k$ is a unit in $\mathbb{Z}_{p^2}$. In this case, the ideals are of the form \[ \langle k p x + l p^2 \rangle, \qquad \text{where } k \in \mathbb{Z}_{p^2}^\times \text{ and } l \in \mathbb{Z}_p^\times\]
        For each fixed $l \in \mathbb{Z}_p^\times$, there are $\varphi(p^2) = p(p - 1)$ such generators. By Lemma \ref{lemmacongideal}, any two ideals of this form are distinct, yielding a total of $p-1$ distinct ideals.
    \end{itemize}
      \case  For $b > 0$, suppose $b$ is divisible by $p$ but not by $p^2$. Then $b = l p$, where $l$ is a unit in $\Z_{p^2}$. This gives rise to the following subcases:
    \begin{itemize}
        \item[(i)] If $a = 0$, the corresponding maximal ideal is $\langle p\rangle$, and there are $\varphi(p^2) = p(p-1)$ such elements.
        
        \item[(ii)] If $a$ is a unit in $\mathbb{Z}_{p^3}$, we obtain ideals of the form 
        \[
        \langle ax + lp \rangle, \quad \text{where } a\in \Z_{p^3}^\times, l \in \Z_{p^2}^\times.
        \]
        Although there are $\phi (p^3)\phi (p^2)$ such generators, not all of them produce distinct ideals. Without loss of generality, taking $a_1=a_2=1$, this reduces to $l_1\equiv l_2 \pmod{p}$ (by Lemma \ref{lemmacongideal}). So $\langle x + lp \rangle \cong \langle x + (l+p)p \rangle$ for any $l\in \Z_{p^2}^\times$, implying that there are exactly $p-1$ distinct principal ideals of the form $\langle x + l p \rangle$, each one is generated by exactly $\phi (p^3)\phi (p^2)/(p-1)=p^3(p-1)$ many distinct elements of ${\Z}_{p^3}[x]/(x^2)$.

        \item[(iii)] If $p^2 \mid a$, the elements are of the form $\langle kp^2x + lp \rangle$, where  $k\in \mathbb{Z}_p^\times$ and $l\in \mathbb{Z}_{p^2}^\times$. However
        \[
          \langle kp^2x + lp \rangle =  \langle p(kpx + l) \rangle = \langle p \rangle,
        \]
        as $kpx + l$ is a unit in ${\Z}_{p^3}[x]/(x^2)$. Note that there are a total of $\phi(p)\phi(p^2)$ many distinct elements in the ring forming this maximal ideal $\langle p \rangle$.
        
        \item[(iv)] Finally $p\mid a$, but $p^2\nmid a$, the elements are of the form $\langle kpx + lp \rangle$, where both $k, l \in \mathbb{Z}_{p^2}^\times$ are units. There are $\phi(p^2)\phi(p^2)$ such elements. However, once again, since 
        \[
        \langle kpx + lp \rangle = \langle p(kx + l) \rangle = \langle p \rangle,
        \]
        These correspond to the same maximal ideal $\langle p \rangle$ already considered. Therefore, the total number of maximal ideals of this type is 
        \[
        (p - 1) + (p-1)p(p-1)+ p^2(p - 1)^2 = p^3(p - 1).
        \] 
    \end{itemize}
\end{mycases}

    


Next, to understand the structure of the cozero-divisor graph of ${\Z}_{p^3}[x]/(x^2)$, we employ a two-stage approach. First, we analyze the lattice of principal ideals of the ring $R=\Z_{p^3}[x]/(x^2) $. This lattice provides a complete map of the ring's algebraic structure, revealing key features such as the maximal ideals $\langle p \rangle$ and $\{\langle x+lp^2 \rangle: l\in \Z_p^\times\}$. Second, we use this lattice to construct the reduced cozero-divisor graph $\Upsilon'[R]$. In this graph, sets of associated elements from the lattice are collapsed into single vertices, simplifying the structure while preserving its essential connectivity.

This construction results in a reduced graph $\Upsilon'[R]$ with exactly $3(p-1)+5=3p+2$ vertices. Crucially, the vertices corresponding to the clusters of ideals $\langle x+lp \rangle$, $\langle x+lp^2 \rangle$, and $\langle xp+lp^2 \rangle$  (for $l\in \Z_p^\times$) each induce a complete subgraph $K_{p-1}$. The entire process, from the ideal lattice to the final reduced graph, is illustrated in Figure \ref{p^3cozeroreduced}.

\begin{figure}[h!]
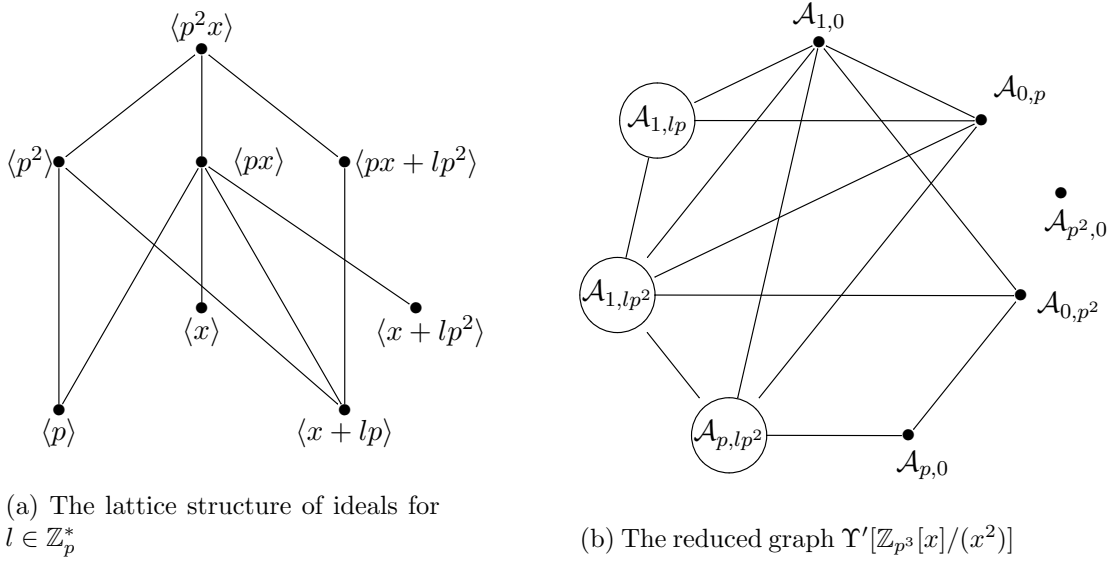

    \centering
     \begin{subfigure}[b]{0.35\textwidth}
         \[ \xygraph{
  !{<0cm,0cm>;<1.9cm,0cm>:<0cm,1.5cm>::}
  !{(0,3)}*{\bullet}="v00" !{(0,3.2)}*{}*{\langle p^2x \rangle}
  !{(-1,2)}*{\bullet}="v10" !{(-1.2,2)}*{}*{\langle p^2 \rangle}
  !{(0,2)}*{\bullet}="v11" !{(0.4,2)}*{}*{\langle px \rangle}
  !{(1,2)}*{\bullet}="v12" !{(1.5,2)}*{}*{\langle px+ lp^2 \rangle}
  !{(0,0.7)}*{\bullet}="v20" !{(0,0.5)}*{}*{\langle x \rangle}
  !{(1.5,0.7)}*{\bullet}="v21" !{(1.6,0.5)}*{}*{\langle x+ lp^2 \rangle}
  !{(-1,-.2)}*{\bullet}="v30" !{(-1,-.4)}*{}*{\langle p \rangle}
  !{(1,-.2)}*{\bullet}="v31" !{(1,-.4)}*{}*{\langle x+lp \rangle}
  "v00"-"v10" "v00"-"v11" "v00"-"v12"
  "v10"-"v30" "v10"-"v31" "v11"-"v20" "v11"-"v21"
  "v12"-"v31" "v11"-"v31"  "v11"-"v30"
   }
\]    
     \subcaption{The lattice structure of ideals for $l\in\Z_{p}^*$}
     \end{subfigure}
     \hspace{1.6 cm}
     \begin{subfigure}[b]{0.35\textwidth}
        \[ \xygraph{
!{<0cm,0cm>;<0cm,1.1cm>:<-1.1cm,0cm>::}
!{(0,0);a(0)**{}?(2.5)}*{\bullet}="v0" !{(0,0);a(0)**{}?(2.8)}*{\A_{1,0}}
!{(0,0);a(51.5)**{}?(2.5)}*{\A_{1,lp}}*{\cir<5mm>{}}="v1" 
!{(0,0);a(102.8)**{}?(2.5)}*{\A_{1,lp^2}}*{\cir<5mm>{}}="v2" 
!{(0,0);a(154.2)**{}?(2.5)}*{\A_{p,lp^2}}*{\cir<5mm>{}}="v3" 
!{(0,0);a(205.6)**{}?(2.5)}*{\bullet}="v4" !{(0,0);a(205.6)**{}?(2.9)}*{\A_{p,0}}
!{(0,0);a(257)**{}?(2.5)}*{\bullet}="v5" !{(0,0);a(257)**{}?(3.1)}*{\A_{0,p^2}}
!{(0,0);a(308.4)**{}?(2.5)}*{\bullet}="v6" !{(0,0);a(308.4)**{}?(3.1)}*{\A_{0,p}}
!{(0,0);a(283)**{}?(3.0)}*{\bullet}="v" !{(0,.06);a(275)**{}?(3.0)}*{\A_{p^2,0}}
"v0"-"v1" "v1"-"v2" "v2"-"v3" "v3"-"v4" "v0"-"v6"
"v0"-"v2" "v1"-"v6"  "v2"-"v6" 
"v0"-"v3" "v3"-"v6" "v0"-"v5" "v2"-"v5" "v4"-"v5" 
}
\]
     \subcaption{The reduced graph $\Upsilon'[\Z_{p^3}[x]/(x^2)]$}
       \end{subfigure}
     \caption{$\Gamma'[\Z_{p^3}[x]/(x^2)]$}  \label{p^3cozeroreduced}
     \end{figure}

We begin by computing the Laplacian spectrum of the graph $\Gamma'(\A_{p,lp^2})$. The graph consists of $p-1$ disjoint copies of a subgraph ${\cal B}_l$, each on $p(p-1)$ vertices. Furthermore, because the elements $\A_{p,l_1p^2}$ and $\A_{p,l_2p^2}$ are incomparable for distinct $l_1,l_2\in \Z_{p-1}^\times$, the vertices corresponding to these $(p-1)$ components form a complete graph $K_{p-1}$ in the reduced graph. Using Theorem \ref{spectrumtheorem}, we get $D(\B_l)= p(p-1)(p-2)$. Thus, the spectrum of the $\Gamma'(\A_{p,lp^2})$ is the following 

\begin{align*}
\Phi_{L}(\Gamma'(\A_{p,lp^2})) &=\bigcup_{l\in \Z_p^\times} \left(D(\B_l)) + (\Phi_L(\B_l)) \setminus \{0\})\right)\bigcup  \Phi(\mathbb{L}(\Upsilon'[\A_{p,lp^2}])) \\ 
&= \begin{pmatrix}
p(p-1)(p-2)   \\
(p-1)(p(p-1)-1)  
\end{pmatrix}  \bigcup 
\begin{pmatrix}
0 & p-1   \\
p-2 & 1
\end{pmatrix} 
\end{align*}

Hence, we can propose the following Lemma.
\begin{lem}
Let $\Gamma'(\A_{p,lp^2})$ denote the cozero-divisor graph corresponding to the ideal $\A_{p,lp^2}$ of the ring $\Z_{p^3}[x]/(x^2)$. Then the Laplacian spectrum of $\Gamma'(\A_{p,lp^2})$ is given by
\[ \Phi_{L}(\Gamma'(\A_{p,lp^2})) = \{0^{[p-2]}, p-1^{[1]},p(p-1)(p-2)^{[(p-1)(p(p-1)-1)]}\}\]
\[\Phi_{L}(\Gamma'(\A_{1,lp}))=\{0^{[p-2]}, p-1^{[1]},p^3(p-1)(p-2)^{[(p-1)(p^3(p-1)-1)]}\}\]
\[\Phi_{L}(\Gamma'(\A_{1,lp^2}))=\{0^{[p-2]}, p-1^{[1]},p^2(p-1)(p-2)^{[(p-1)(p^2(p-1)-1)]}\}\]
\end{lem}

\begin{thrm}
The Laplacian spectrum of $\Gamma'(\mathbb{Z}_{p^3}[x]/(x^2))$ is given by:
\[
\begin{aligned}
&\left\{ 
0^{[p-1]},\ 
p^3(p^2-1)^{[p^2(p-1)-1]},\ 
p^4(p-1)^{[(p-1)(p^3(p-1)-1)]},\right.\\
&2p^3(p-1)^{[p-3]},\ 
(p-1)(2p^3+1)^{[1]},\ 
p^3(p^2-1)^{[(p-1)(p^2(p-1)-1)]},\\
&p^2(p-1)(p^2+2)^{[p-3]},\ 
{(p-1)(p^2+1)^2}^{[1]},\ 
p^2(p-1)(2p+1)^{[(p-1)(p(p-1)-1)]},\\
& 2p(p-1)(p^2+1)^{[p-3]},\ 
(p-1)(p^3+2p+1)^{[1]},\ 
p(p-1)(p^3+p-1)^{[p^3(p-1)-1]},\\
&\left.p^2(p^2-1)^{[p(p-1)-1]},\
p^2(p-1)^{[p(p-1)-1]},
\right\}\cup\{\lambda_i^{[1]}\}_{i=1}^7, 
\end{aligned}
\]
where $\{\lambda_i\}_{i=1}^7$ are the seven eigenvalues arising from the Laplacian matrix of the reduced graph  $\Upsilon'[\Z_{p^3}[x]/(x^2)]$.
\end{thrm}

\begin{proof}
The cozero-divisor graph decomposes as: \[\Gamma'(\Z_{p^2}[x]/(x^2))=\Upsilon'[\Gamma'(\A_{1,0}),\Gamma'(\A_{1,lp}),\Gamma'(\A_{1,lp^2}),\Gamma'(\A_{p,lp^2}), \Gamma'(\A_{p,0}),\Gamma'(\A_{0,p^2}),\Gamma'(\A_{0,p}), \Gamma'(\A_{p^2x,0})]\]
The vertex counts and degrees in the reduced graph are:
\[
\begin{aligned}
n_{1,0} &= p^2(p - 1), & D_{1,0} &= p^3(p^2-1), \\
n_{1,lp} &= p^3(p - 1)^2, & D_{1,lp} &= 2p^3(p-1), \\
n_{1,lp^2} &= p^2(p - 1)^2, & D_{1,lp^2} &= p^2(p-1)(p^2+2), \\
n_{p,lp^2} &= p(p - 1)^2, & D_{p,lp^2} &= 2p(p-1)(p^2+1), \\
n_{p,0} &= p(p - 1), & D_{p,0} &= p^2(p-1), \\
n_{0,p^2} &= p(p - 1), & D_{0,p^2} &= p^2(p^2-1), \\
n_{0,p} &= p^3(p - 1), & D_{0,p} &= p(p-1)(p^3+p-1), \\
n_{p^2x,0} &= p-1.
\end{aligned}
\]
Now the spectrum for the vertex $v_{1,lp}$ in the reduced graph $\Upsilon'[\Z_{p^2}[x]/(x^2)])$, $D_{1,p} = p^3(p-1)^2+p^2(p-1)+p^3(p-1) +p(p-1)+p(p-1)^2 = 2p^3(p-1)$, and consequently, 
\begin{align*}
\left(D_{1,lp} + (\Phi_L(\Gamma'(\A_{1,lp})) \setminus \{0\})\right)
= \begin{pmatrix}
p^4(p-1)          & 2p^3(p-1)  & (p-1)(2p^3+1) \\
(p-1)(p^3(p-1)-1) & p-3        & 1 
\end{pmatrix} 
\end{align*}

Similarly 
\begin{align*}
\left(D_{1,lp^2} + (\Phi_L(\Gamma'(\A_{1,lp^2})) \setminus \{0\})\right)
= \begin{pmatrix}
p^3(p^2-1) & p^2(p-1)(p^2+2)  & (p-1)(p^2+1)^2 \\
(p-1)(p^2(p-1)-1) & p-3 & 1 
\end{pmatrix} 
\end{align*}

\begin{align*}
\left(D_{p,lp^2} + (\Phi_L(\Gamma'(\A_{p,lp^2})) \setminus \{0\})\right)
= \begin{pmatrix}
p^2(p-1)(2p+1) & 2p(p-1)(p^2+1)  & (p-1)(p^3+2p+1) \\
(p-1)(p(p-1)-1) & p-3 & 1 
\end{pmatrix} 
\end{align*}
The remaining eigenvalues are the eigenvalue of the matrix $\Phi(\mathbb{L}(\Upsilon'[\Z_{p^3}[x]/(x^2)]))$, where $\Upsilon'[\Z_{p^3}[x]/(x^2)]$ is illustrated in Figure \ref{p^3cozeroreduced}. 

\[
L(\Gamma) = \begin{pmatrix}
D_0 & -p^3(p-1)^2 & -p^2(p-1)^2 & -p(p-1)^2 & 0 & -p(p-1) & -p^3(p-1) \\
-p^2(p-1) & D_1 & -p^2(p-1)^2 & 0 & 0 & 0 & -p^3(p-1) \\
-p^2(p-1) & -p^3(p-1)^2 & D_2 & -p(p-1)^2 & 0 & -p(p-1) & -p^3(p-1) \\
-p^2(p-1) & 0 & -p^2(p-1)^2 & D_3 & -p(p-1) & -p(p-1) & -p^3(p-1) \\
0 & 0 & 0 & -p(p-1)^2 & D_4 & -p(p-1) & 0 \\
-p^2(p-1) & 0 & -p^2(p-1)^2 & -p(p-1)^2 & -p(p-1) & D_5 & 0 \\
-p^2(p-1) & -p^3(p-1)^2 & -p^2(p-1)^2 & -p(p-1)^2 & 0 & 0 & D_6
\end{pmatrix}
\]
\end{proof}

\subsection{Mixed prime powers: $n = p^2q$}
It can be noted that $\mid V(\Gamma'(\Z_{p^2}[x]/(X^2))) \mid =p^2q(p^2q-\phi (p^2q))-1$. Just like $\Z_{pq}[x]/(X^2)$, every principal ideal in $\Z_{p^2q}[x]/(X^2)$ corresponds uniquely to an ordered pair of principal ideals, one from each component ring $\Z_{p^2}[x]/(X^2)$, and the other one from $\Z_{q}[x]/(X^2)$. For the local rings $\Z_{p^2}[x]/(X^2)$ and $\Z_{q}[x]/(X^2)$, the principal ideals can be categorized as follows:

\[
\begin{aligned}
\text{In } \mathbb{Z}_{p^2}[x]/(x^2):\ & I_1 = \langle 0 \rangle,\quad
I_2 = \langle px \rangle,\quad
I_3 = \langle x \rangle,\quad 
I_4 = \{\langle p \rangle \},\quad\\&
I_{l+4} = \{\langle x + lp\rangle: l\in \Z_{p}^\times\},\quad
I_{p+4} = \{\langle a x + b \rangle : b \in \mathbb{Z}_p^\times\}, \quad\\[4pt]
\text{In } \mathbb{Z}_{q}[x]/(x^2):\ &
J_1 = \langle 0 \rangle,\quad
J_2 = \langle x \rangle,\quad
J_3 = \{\langle c x + d \rangle : b \in \mathbb{Z}_q^\times\}.
\end{aligned}
\]
The corresponding cardinalities are:
\[
\mid I_1\mid  = \mid J_1\mid = 1,\quad
\mid I_2\mid  = p-1,\quad
\mid J_2\mid = q-1,\quad
\mid I_{p+4}\mid = p^3(p-1),\quad
\mid J_3\mid  = q(q-1),\quad
\mid I_{l+2}\mid = p(p-1)\quad
\]
for $l\in \{1,2,\cdots,p+1\}$. Let $a_{p^2},b_{p^2} \in \mathbb{Z}_{p^2}$ and $a_q,b_q \in \Z_q$ denote the images of $a,b \in \mathbb{Z}_{p^2q}$ modulo $p^2$ and $q$, respectively. Each principal ideal $\langle a x + b \rangle$ in $\Z_{pq}[x]/(x^2)$ can then be analyzed via the pair of ideals it generates in the component rings. There are $3(p+4)$ possible combinations of $(I_i, J_j)$, which we have listed in Table \ref{tablecombinationp^2q}. Note that $J_1\subset J_2 \subset J_3$, and $I_1\subset I_2 \subset I_k \subset I_{p+4}$ where $k\in\{3,5,\cdots,p+3\}$. Also all the ideals in this $\{I_k\}_{k=3}^{p+3}$ are mutually incomparable. 
Now we define the following sets: 
$${\cal A}_{i,j}=\{cx+d\in \Z_{pq}[x]/(x^2): \langle cx+d \rangle\in I_i\times J_j\}$$

\begin{table}[h!]
\centering
\renewcommand{\arraystretch}{1.2}
\begin{tabular}{|c|c|c|c|}
\hline
\textbf{Ideal in } $\mathbb{Z}_{p^2}[x]/(x^2)$ & 
\textbf{Ideal in } $\mathbb{Z}_{q}[x]/(x^2)$ & 
\textbf{Notation in } $R$ & 
\textbf{Cardinality } $\mid I_i\mid\mid J_j\vert $ \\ \hline

& $J_1 = \langle 0 \rangle$ & $(I_1, J_1)$ & $1$ \\ 
$I_1 = \langle 0 \rangle$  & $J_2 = \langle x \rangle$ & $(I_1, J_2)$ & $q - 1$ \\ 
& $J_3 = \langle c x + d \rangle,\ d \in \Z_q^\times$ & $(I_1, J_3)$ & $q(q - 1)$ \\ \hline

 & $J_1$ & $(I_2, J_1)$ & $p - 1$ \\ 
$I_2 = \langle p x \rangle$ & $J_2$ & $(I_2, J_2)$ & $(p - 1)(q - 1)$ \\ 
 & $J_3$ & $(I_2, J_3)$ & $(p - 1)q(q - 1)$ \\ \hline

 & $J_1$ & $(I_3, J_1)$ & $p(p - 1)$ \\ 
$I_3 = \langle x \rangle$ & $J_2$ & $(I_3, J_2)$ & $p(p - 1)(q - 1)$ \\ 
& $J_3$ & $(I_3, J_3)$ & $p(p - 1)q(q - 1)$ \\ \hline

 & $J_1$ & $(I_4, J_1)$ & $p(p - 1)$ \\ 
$I_4 = \langle p \rangle$ & $J_2$ & $(I_4, J_2)$ & $p(p - 1)(q - 1)$ \\ 
 & $J_3$ & $(I_4, J_3)$ & $p(p - 1)q(q - 1)$ \\ \hline

 & $J_1$ & $(I_{l+5}, J_1)$ & $p(p - 1)$ \\ 
$I_{l+4} = \langle x + l p \rangle,\ l \in \Z_p^\times$& $J_2$ & $(I_{l+5}, J_2)$ & $p(p - 1)(q - 1)$ \\ 
& $J_3$ & $(I_{l+5}, J_3)$ & $p(p - 1)q(q - 1)$ \\ \hline

 & $J_1$ & $(I_{p+4}, J_1)$ & $p^3(p - 1)$ \\ 
$I_{p+4} = \langle a x + b \rangle,\ b \in \Z_p^\times$ & $J_2$ & $(I_{p+4}, J_2)$ & $p^3(p - 1)(q - 1)$ \\ 
 & $J_3$ & $(I_{p+4}, J_3)$ & $p^3(p - 1)q(q - 1)$ \\ \hline

\end{tabular}\caption{All possible combinations of $(I_i, J_j)$ in $\mathbb{Z}_{p^2q}[x]/(x^2)$ and their cardinalities.}\label{tablecombinationp^2q}
\end{table}

\begin{figure}[h!]
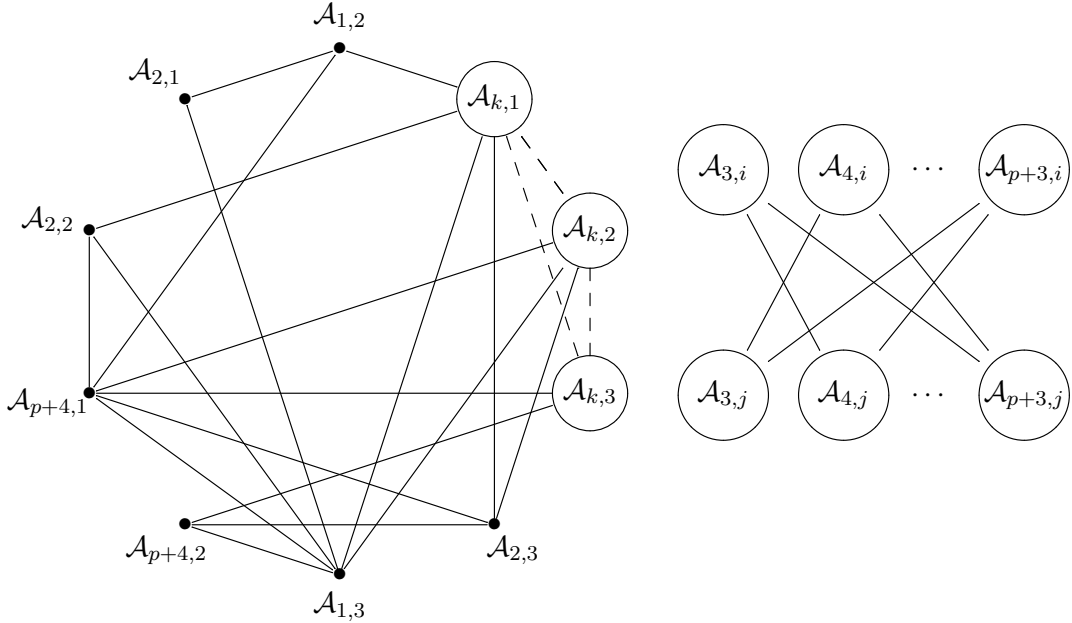

    \centering
     \begin{subfigure}[b]{0.35\textwidth}
         \[ \xygraph{
!{<0cm,0cm>;<0cm,1.4cm>:<1.4cm,0cm>::}
!{(0,0);a(0)**{}?(2.5)}*{\bullet}="v0" !{(0,0);a(0)**{}?(2.8)}*{\A_{1,2}}
!{(0,0);a(36)**{}?(2.5)}*{\A_{k,1}}*{\cir<5mm>{}}="v1" 
!{(0,0);a(72)**{}?(2.5)}*{\A_{k,2}}*{\cir<5mm>{}}="v2" 
!{(0,0);a(108)**{}?(2.5)}*{\A_{k,3}}*{\cir<5mm>{}}="v3" 
!{(0,0);a(144)**{}?(2.5)}*{\bullet}="v4" !{(0,0);a(144)**{}?(2.8)}*{\A_{2,3}}
!{(0,0);a(180)**{}?(2.5)}*{\bullet}="v5" !{(0,0);a(180)**{}?(2.8)}*{\A_{1,3}}
!{(0,0);a(216)**{}?(2.5)}*{\bullet}="v6" !{(0,0);a(216)**{}?(2.8)}*{\A_{p+4,2}}
!{(0,0);a(252)**{}?(2.5)}*{\bullet}="v7" !{(0,.06);a(252)**{}?(2.8)}*{\A_{p+4,1}}
!{(0,0);a(288)**{}?(2.5)}*{\bullet}="v8" !{(0,.06);a(288)**{}?(2.8)}*{\A_{2,2}}
!{(0,0);a(324)**{}?(2.5)}*{\bullet}="v9" !{(0,.06);a(324)**{}?(2.8)}*{\A_{2,1}}
"v0"-"v1" "v0"-"v9" "v0"-"v7" 
"v1"-"v4" "v1"-@{--}"v2" "v1"-@{--}"v3" "v1"-"v5" "v1"-"v8" 
"v2"-@{--}"v1" "v2"-@{--}"v3" "v2"-"v4" "v2"-"v5" "v2"-"v7" 
"v3"-"v6" "v3"-"v7"
"v4"-"v6" "v4"-"v7"
"v5"-"v6" "v5"-"v7" "v5"-"v8" "v5"-"v9"
"v7"-"v8"
}
\]    
     \subcaption{The partially-reduced graph $\bar\Upsilon'[\Z_{p^2q}[x]/(x^2)]$}\label{p^2qcozeroA}
     \end{subfigure}
     \hspace{2.6 cm}
     \begin{subfigure}[b]{0.35\textwidth}
        \[ \xygraph{
!{<0cm,0cm>;<1.6 cm,0cm>:<0cm,1cm>::}
!{(-1,5) }*{\A_{3,i}}*{\cir<6mm>{}}="u1" 
!{(0,5) }*{\A_{4,i}}*{\cir<6mm>{}}="u2" 
!{(0.7,5)}*{\cdots}="cdots"
!{(1.5,5) }*{\A_{p+3,i}}*{\cir<6mm>{}}="u3" 
!{(-1,2) }*{\A_{3,j}}*{\cir<6mm>{}}="v1" 
!{(0,2)}*{\A_{4,j}}*{\cir<6mm>{}}="v2" 
!{(0.7,2)}*{\cdots}="cdots"
!{(1.5,2) }*{\A_{p+3,j}}*{\cir<6mm>{}}="v3" 
!{(1.5,-1) }*{}*{}="v4" 
"u1"-"v2" "u1"-"v3" "u2"-"v1" "u2"-"v3" "u3"-"v1" "u3"-"v2"  
}
\]
     \subcaption{The zoomed view of adjacency of $\A_{k,i}$ and $\A_{k,j}$, for $i\ne j$, $i,j\in\{1,2,3\}$}\label{p^2qcozeroB}
     \end{subfigure}
     \caption{The structure of the cozero-divisor graph $\Gamma'[\Z_{p^2q}[x]/(x^2)]$}
     \label{p^2qcozero}
\end{figure}

The cardinality of each set $\mathcal{A}_{ij}$ is given by the product of the cardinality of their corresponding ideals. Clearly, no two vertices within the same set $\mathcal{A}_{i,j}$ are adjacent, implying that the induced subgraph $\Gamma'(\mathcal{A}_{i,j})$ is isomorphic to the null graph $\overline{K}_{\mid I_i \mid \mid J_j \mid}$. We omit the sets $\mathcal{A}_{1,1}$ and $\mathcal{A}_{p+4,3}$, as they correspond to the zero and unit elements, respectively. The following theorem is analogous to Theorem~\ref{connectingtheorempqr}, and therefore we omit its proof.


\begin{thrm}
In the cozero-divisor graph $\Gamma'(\mathbb{Z}_{p^2q}[x]/(x^2))$, every vertex in $\mathcal{A}_{i_1,j_1}$ is adjacent to every vertex in $\mathcal{A}_{i_2,j_2}$ if and only if the index pairs $(i_1,j_1)$ and $(i_2,j_2)$ satisfy at least one of the following conditions:
\begin{enumerate}
    \item $i_1 \neq i_2$ and $i_1, i_2 \in \{3,4,\dots,p+3\}$;
    \item $i_1 < i_2$ and $j_1 > j_2$, where $i_1, i_2 \in \{1,2,\dots,p+4\}$ with $i_1 \neq i_2$, and $j_1, j_2 \in \{1,2,3\}$ with $j_1 \neq j_2$.
\end{enumerate}
Moreover, vertices from the same set $\mathcal{A}_{i,j}$ are never adjacent to each other.
\end{thrm}

The cozero-divisor graph can be constructed from a structure closely resembling the reduced graph (Figure~\ref{p^2qcozeroA}), with one key distinction: in this graph, vertices within the same $k$-family $\mathcal{A}_{k,i}$ and $\mathcal{A}_{k,j}$ 
are never adjacent for any $k \in \{3,4,\dots,p+4\}$, hence we connected those with dotted edge. The adjacency (dotted edge) between $\mathcal{A}_{k,i}$ and $\mathcal{A}_{k,j}$ is illustrated in Figure~\ref{p^2qcozeroB}. We refer to this modified structure as the \emph{partially-reduced graph} $\bar\Upsilon'[\Z_{p^2q}[x]/(x^2)]$.

\begin{thrm}
Let \( p \) and \( q \) be primes. The Laplacian spectrum of the cozero-divisor graph \( \Gamma'(\mathbb{Z}_{p^2q}[x]/(x^2)) \) is given by
\[
\begin{aligned}
&\left\{ 
(p^4-1)^{[q-2]},\ 
p(p^2q^2-pq^2+q^2-1)^{[p(p-1)-1]},\ 
p(p^2q^2-pq^2+p^3-p^2+q^2-q)^{[p(p-1)(q-1)-1]},\right.\\
&p^2q(p-1)(p+q)^{[pq(p-1)(q-1)-1]},\ 
pq(p^3-1)^{[q(p-1)(q-1)-1]},\ 
q(p-1)(p^3+p(q-1)+1)^{[q(q-1)-1]},\\
&p^3q(q-1)^{[p^3(p-1)(q-1)-1]},\ 
p^3(q^2-1)^{[p^3(p-1)-1]},\ 
(p^4-p+q^2-q)^{[(p-1)(q-1)-1]},\ 
(q^2-1)^{[p-2]},\\
&\left. \vphantom{p^2q(p-1)(p+q)^{[pq(p-1)(q-1)-1]}}
\right\} \cup \{\lambda_i^{[1]}\}_{i=1}^{3p+10},
\end{aligned}
\]
where $\{\lambda_i\}_{i=1}^{3p+10}$ are the eigenvalues arising from the Laplacian matrix of the reduced graph  $\Upsilon'[\Z_{p^2q}[x]/(x^2)]$.
\end{thrm}

\begin{proof}
In order to calculate the Laplacian spectrum of the cozero-divisor graph $\Gamma'(\mathbb{Z}_{p^2q}[x]/(x^2))$, first we consider the set $\A_{k,j}$, for all $k\in \{3,4,\cdots,p+3\}, j\in \{1,2,3\}$. Note that any vertex in $\A_{k,j}$ is connected with $p(p(p-1)+p(q-1)(p-1)+pq(p-1)(q-1))=p^2q^2(p-1)$ many vertices inside $\A_{k,j}$. Next, we calculate the total adjacency of each of the vertices in each cluster $\A_{k,j}$ for all $k\in \{1,2,\cdots,p+4\}$, and $ j\in \{1,2,3\}$.
A straightforward computation yields the following simplified expressions:
\begin{align*}
    D_{1,2} &= p^4 - 1, \quad &
    D_{k,1} &= p(p^2q^2 - pq^2 + q^2 - 1) \\
    D_{k,2} &= p(p^2q^2 - pq^2 + p^3 - p^2 + q^2 - q) \quad &
    D_{k,3} &= p^2q(p-1)(p + q) \\
    D_{2,3} &= pq(p^3 - 1) \quad &
    D_{1,3} &= (p-1)q(p^3 + p^2 + p + 1) \\
    D_{p+4,2} &= p^3q(q - 1) \quad &
    D_{p+4,1} &= p^3(q^2 - 1) \\
    D_{2,2} &= p^4 - p + q^2 - q \quad &
    D_{2,1} &= q^2 - 1
\end{align*}
Finally,
\begin{align*}
\Phi_{L}(\Gamma'(\Z_{p^2q}[x]/(x^2))) =&\bigcup_{(i,j)\in S} \left(D_{i,j} + (\Phi_L(\Gamma'(\A_{i,j})) \setminus \{0\})\right)\bigcup  \Phi(\mathbb{L}(\Upsilon'[\Z_{p^2q}[x]/(x^2)])) \\ 
&= \begin{pmatrix}
p^4-1 & p(p^2q^2-pq^2+q^2-1) & p(p^2q^2-pq^2+p^3-p^2+q^2-q) \\
q-2 & p(p-1)-1  & p(p-1)(q-1)-1  
\end{pmatrix} \\
&  \bigcup  \begin{pmatrix}
p^2q(p-1)(p+q) & pq(p^3-1)  & q(p-1)(p^3+p(q-1)+1) \\
pq(p-1)(q-1)-1 & q(p-1)(q-1)-1 &  q(q-1)-1 
\end{pmatrix} \\
&\bigcup \begin{pmatrix}
p^3q(q-1) & p^3(q^2-1) &  p^4-p+q^2-q & q^2-1\\
p^3(p-1)(q-1)-1 & p^3(p-1)-1 & (p-1)(q-1)-1  & p-2
\end{pmatrix} \\ &
\bigcup \Phi(\mathbb{L}(\Upsilon'[\Z_{p^2}[x]/(x^2)])).
\end{align*}

The spectrum is completed by computing the eigenvalues of the Laplacian matrix of the graph $\Upsilon'[\Z_{p^2q}[x]/(x^2)] $. Note that the dimension is this matrix is linear in $p$, specifically $3(p+1)+7=3p+10$. This provides a dramatic reduction in complexity, as the original adjacency matrix of the graph has a number of vertices on the order of $p^2q$.
\end{proof}

We finish this paper by exploring the connectivity of the cozero-divisor graph $\Gamma'(\Z_n[x]/(x^2))$. 
\begin{prop}
The cozero-divisor graph $\Gamma'(\Z_n[x]/(x^2))$ is disconnected if and only if $n$ is a prime power.
\end{prop}

\begin{proof}
First, suppose $n = p^k$, a prime power. Consider the principal ideal $I = \langle p^{k-1}x \rangle$ in $\Z_n[x]/(x^2)$. For any element $ax + b \in \Z_n[x]/(x^2)$ with $\langle ax + b \rangle \neq I$, we analyze the containment relations. For any ideal $\langle ax+b \rangle$, if $b=0$, then there are two possibilities.
\begin{enumerate}
    \item $p^{k-1}\mid a$, then $\langle ax\rangle \subseteq \langle p^{k-1}x\rangle$
    \item $p^{k-1}\nmid a$, then $\langle p^{k-1}x\rangle \subseteq  \langle ax\rangle $
\end{enumerate}

In both cases, $\langle ax\rangle$ and $\langle p^{k-1}x \rangle$ are comparable, hence, $\langle p^{k-1}x \rangle$ and $\langle ax\rangle$  are not adjacent in $\Gamma'(\Z_{p^k}[x]/(x^2))$. Now let us consider the remaining case, $b=jp^l$ for $0\le l <k$, and $j\ne 0$ such that $(j,p)=1$. Then note 
$(ax+jp^l)(j^{-1}p^{k-1-l}x)= p^{k-1}x$. Hence $\langle p^{k-1}x \rangle \subseteq \langle ax + b \rangle$, and consequently $\langle p^{k-1}x \rangle$ is not adjacent to any other vertex in $\Gamma'(\Z_n[x]/(x^2))$. Therefore, the vertex corresponding to $\langle p^{k-1}x \rangle$ forms an isolated vertex, making the graph disconnected.

Conversely, we assume that $n$ is not a prime power and show that the cozero-divisor graph obtained is connected. Let $n=\prod_{i=1}^r p_i^{\alpha_i}$ be its prime factorization, with $r\ge 2$. Then from the isomorphism of rings, 
\[
\mathbb{Z}_{n}[x]/(x^{2}) \cong \prod_{i=1}^{r} \mathbb{Z}_{p_{i}^{\alpha_{i}}}[x]/(x^{2}),
\]

For each index \(i\) let $J_{i} := \langle p_{i}^{\alpha_{i}-1}x \rangle$. Now we define the ideals $U_i$ in $\Z_{n}[x]/(x^{2})$ corresponding to the ideal $J_i$ in the $i^{\rm th}$ component.

\[
U_{i} = (0, \dots, \underbrace{\langle p_{i}^{\alpha_{i}-1} x\rangle}_{i^\text{th} \text{position}}, \dots, 0)
\]
Note that for \(i\neq j\) the ideals \(U_{i}\) and \(U_{j}\) are incomparable (\(U_{i} \not \subseteq U_{j}, U_{j} \not \subseteq U_{i}\)). Hence every \(U_{i}\) and \(U_{j}\) $(i\ne j)$ are adjacent in \(\Gamma^{\prime}(\mathbb{Z}_{n}[x]/(x^{2}))\). In particular, the subgraph induced by \(\{U_{1}, \ldots, U_{r}\}\) is a clique. Now for any vertex $f=ax+b\in V(\Gamma'(\Z_n[x]/(x^2)))$, we show that it is connected with this clique. We write $f=(f_1,f_2,\cdots,f_r)$, where $f_i=a_ix+b_i$, $a\equiv a_i \pmod {p_i^{\alpha_i}}, b\equiv  b_i \pmod {p_i^{\alpha_i}}$. 
We consider the following two cases:
\begin{enumerate}
    \item $f$ has at least one zero component. In that case note that there also exist some $f_j\ne 0$ (otherwise $f$ is the zero element, not part of the set $V(\Gamma'(\Z_n[x]/(x^2)))$).Without loss of generality, let us assume that $f_1=0$, and $f_2\ne 0$. Now there exists a vertex $g=(g_1,g_2,\cdots, g_r)$ such that $g_1$ is a unit (i.e. $\langle g_1 \rangle =\Z_{p_1^{\alpha_1}}[x]/(x^2)$) and $g_2=0$. Thus $\langle f\rangle \not \subseteq \langle g\rangle$, and $\langle g\rangle \not \subseteq \langle f\rangle$, hence  $f\sim g$. On the other hand, $U_2$ and $\langle g\rangle$ are incomparable as well, which implies that $g\sim U_2$, establishing the connectivity of $f$ with the clique. 
       
    \item Next we consider $f$ has no zero components. Clearly, at least one of the $f_i$ must not be a unit, as otherwise $f$ will be a unit, and hence not in the vertex set. Without loss of generality, let us assume $f_1$ is non-unit (i.e.$\langle f_1 \rangle \subsetneqq \Z_{p_1^{\alpha_1}}[x]/(x^2)$. Once again we consider a vertex $g=(g_1,g_2,\cdots, g_r)$ such that $g_1$ is a unit (i.e. $\langle g_1 \rangle =\Z_{p_1^{\alpha_1}}[x]/(x^2)$) and $g_2=0$. It is again evident that $f\sim g\sim U_2$.
\end{enumerate}
In all cases, $f=ax+b$ is adjacent to some vertex in the clique ${U_1, \dots, U_r}$, so $\Gamma'(\Z_n[x]/(x^2))$ is connected.
\end{proof}

\section*{Declarations}
\noindent \textbf{Conflict of interest:} The authors have no relevant financial or non-financial interests to disclose

\noindent \textbf{Data Availability:}  No data were used or generated as part of this work.

\noindent \textbf{Funding sources:} 
This research did not receive any specific grant from funding agencies in the public, commercial, or not-for-profit sectors.

\section*{Acknowledgments}
The authors would like to thank Kaleb Brunhoeber (Orcid ID: 0000-0002-7169-5773) for his valuable assistance in generating Figure \ref{figcozeropqr} using Python.


\bibliographystyle{amsplain}

\end{document}